\documentclass[11pt]{article}

\usepackage{amsmath, amsthm}
\usepackage{amsmath, amsfonts}
\usepackage{amsmath, amssymb}
\usepackage{amsmath}
\usepackage{graphics}
\usepackage{graphicx}
\usepackage{color}
\usepackage{xcolor}
\usepackage{hyperref}
\usepackage[all]{xy}

\newtheorem{definition-lemma}[theorem]{Definition/Lemma}
\newtheorem{definition-explanation}[theorem]{Definition/Explanation}

\newtheorem{explanation-definition}[theorem]{Explanation/Definition}
\newtheorem{definition-fact}[theorem]{Definition/Fact}
\newtheorem{definition-notation}[theorem]{Definition/Notation}
\newtheorem{definition-conjecture}[theorem]{Definition/Conjecture}

\newtheorem{lemma-definition}[theorem]{Lemma/Definition}

\newtheorem{remark-notation}[theorem]{\it Remark/Notation}

\newtheorem{application-lemma}[theorem]{Application/Lemma}

\newtheorem{example-definition}[theorem]{Example/Definition}

\newtheorem{definition-prototype}[theorem]{Definition-Prototype}

\numberwithin{equation}{subsection}

\newtheorem{stheorem}{Theorem}[section]
\newtheorem{sdefinition}[stheorem]{Definition}
\newtheorem{sdefinition-lemma}[stheorem]{Definition/Lemma}
\newtheorem{sdefinition-explanation}[stheorem]{Definition/Explanation}

\newtheorem{sexplanation-definition}[stheorem]{Explanation/Definition}
\newtheorem{sdefinition-fact}[stheorem]{Definition/Fact}
\newtheorem{sdefinition-notation}[stheorem]{Definition/Notation}
\newtheorem{sdefinition-conjecture}[stheorem]{Definition/Conjecture}
\newtheorem{slemma}[stheorem]{Lemma}
\newtheorem{slemma-definition}[stheorem]{Lemma/Definition}

\newtheorem{sremark}[stheorem]{\it Remark}
\newtheorem{sremark-notation}[stheorem]{\it Remark/Notation}

\newtheorem{sapplication-lemma}[stheorem]{Application/Lemma}

\newtheorem{sexample-definition}[stheorem]{Example/Definition}

\newtheorem{snotation}[stheorem]{Notation}

\newtheorem{sdefinition-prototype}[stheorem]{Definition-Prototype}

\newtheorem{ssdefinition-lemma}[sstheorem]{Definition/Lemma}
\newtheorem{ssdefinition-explanation}[sstheorem]{Definition/Explanation}

\newtheorem{ssexplanation-definition}[sstheorem]{Explanation/Definition}
\newtheorem{ssdefinition-fact}[sstheorem]{Definition/Fact}
\newtheorem{ssdefinition-notation}[sstheorem]{Definition/Notation}
\newtheorem{ssdefinition-conjecture}[sstheorem]{Definition/Conjecture}

\newtheorem{sslemma-definition}[sstheorem]{Lemma/Definition}

\newtheorem{ssremark-notation}[sstheorem]{\it Remark/Notation}

\newtheorem{ssapplication-lemma}[sstheorem]{Application/Lemma}

\newtheorem{ssexample-definition}[sstheorem]{Example/Definition}

\newtheorem{ssdefinition-prototype}[sstheorem]{Definition-Prototype}

\definecolor{gray02}{cmyk}{0, 0, 0, .02}
\definecolor{gray05}{cmyk}{0, 0, 0, .05}
\definecolor{gray07}{cmyk}{0, 0, 0, .07}

\definecolor{gray10}{cmyk}{0, 0, 0, .1}
\definecolor{gray15}{cmyk}{0, 0, 0, .15}
\definecolor{gray20}{cmyk}{0, 0, 0, .2}
\definecolor{gray25}{cmyk}{0, 0, 0, .25}
\definecolor{gray30}{cmyk}{0, 0, 0, .3}
\definecolor{gray35}{cmyk}{0, 0, 0, .35}
\definecolor{gray40}{cmyk}{0, 0, 0, .4}
\definecolor{gray45}{cmyk}{0, 0, 0, .45}
\definecolor{gray50}{cmyk}{0, 0, 0, .5}
\definecolor{gray55}{cmyk}{0, 0, 0, .55}
\definecolor{gray60}{cmyk}{0, 0, 0, .6}
\definecolor{gray65}{cmyk}{0, 0, 0, .65}
\definecolor{gray70}{cmyk}{0, 0, 0, .7}
\definecolor{gray75}{cmyk}{0, 0, 0, .75}
\definecolor{gray80}{cmyk}{0, 0, 0, .8}
\definecolor{gray85}{cmyk}{0, 0, 0, .85}
\definecolor{gray90}{cmyk}{0, 0, 0, .9}
\definecolor{gray95}{cmyk}{0, 0, 0, .95}

\newcommand{\CP}{{\Bbb C}{\rm P}}

\newcommand{\Fl}{{\mbox{\it Fl}\,}}

\newcommand{\Gr}{\mbox{\it Gr}\,}

\newcommand{\Id}{\mbox{\it Id}\,}

\newcommand{\Span}{\mbox{\it Span}\,}

\newcommand{\SU}{\mbox{\it SU}}

\newcommand{\determinant}{\mbox{\it det}\,}

\newcommand{\boldx}{\mbox{\boldmath $x$}}
  
\newcommand{\boldy}{\mbox{\boldmath $y$}}

\newcommand{\boldz}{\mbox{\boldmath $z$}}

\newcommand{\underlineboldI}{{\underline{\mathbf{I}}}}  
\newcommand{\underlineboldII}{{\underline{\mathbf{I}\!\mathbf{I}}}}

\newcommand{\LARGEdot}{\raisebox{-.4ex}{\LARGE $\cdot$}}

\newcommand{\tinybullet}{{\raisebox{.2ex}{\tiny $\bullet$}}}							
																				
\begin{document}

\enlargethispage{24cm}

\begin{titlepage}

$ $

\vspace{-1.5cm} % Re: -1.5cm for PC; -2.5cm for UT-Math-system

\noindent\hspace{-1cm}
\parbox{6cm}{\small March 2022}\
   \hspace{7cm}\
   \parbox[t]{6cm}{\small\flushleft
                arXiv:yymm.nnnnn [math.AG] \\
                NCS(2): flag variety  % \\ $\mbox{\hspace{3.8em}}$  	??????				                 
				}

%\vspace{2em}
\vspace{2cm}

%title
\centerline{\Large\bf Soft noncommutative flag schemes}
% \vspace{1ex}
% \centerline{\large\bf
%   ???????????}
% \vspace{1ex}
% \centerline{\large\bf
%   ?????????????????}
% \vspace{1ex}
% \centerline{\large\bf
%   ?????????}
 
% end-title

\bigskip

%\bigskip
%\vspace{2.4em}
\vspace{3em}

%authors-'n-addresses
\centerline{\large
  Chien-Hao Liu  \hspace{1ex} and \hspace{1ex}
  Shing-Tung Yau
}

%\vspace{2em}
%\vspace{3em}
\vspace{4em}

%abstract%
\begin{quotation}
\centerline{\bf Abstract}
\vspace{0.3cm}

\baselineskip 12pt  %13pt for [12pt] style
{\small
  The construction of soft noncommutative schemes via toric geometry in
    arXiv:2108.05328 [math.AG] (D(15.1), NCS(1))
   can be generalized and applied to
    a commutative scheme with a distinguished atlas of reasonably good affine local coordinate charts.
 In the current notes we carry out this exercise for flag varieties.
 } % endsmall
\end{quotation}

%\smallskip
%\bigskip
\vspace{19em}

\baselineskip 12pt
{\footnotesize
} %end-footnotesize

%\smallskip
 \bigskip

\noindent {\small MSC number 2020: 14M15; 14A22, 14A15; 81T30
} % end-small

% \smallskip
\bigskip

\baselineskip 10pt
% Re: 11pt for [11pt] style; 12pt for [12pt] style
{\scriptsize
\noindent{\bf Acknowledgements.}
 We thank
   Tsung-Ju Lee and Yun Shi
      for organizing the Algebraic Geometry in String Theory Seminar at CMSA that nourishes us intellectually.	
 S.-T.\ Yau thanks also Tsinghua University, China, for hospitality.
 C.-H.L.\ thanks in addition
   Fei Xie
      for discussions of an issue in Noncommutative Algebraic Geometry, fall 2021;
   Fernando Quevedo,,
   Percy Liang, Dorsa Sadigh,,
   Leonard Bernstein (1918-1990),,
   Xiang Luo
      for lecture series/open courses:
   John Madore,,
   William Christ, Richard DeLone (1928-1984), Vernon Kliewer (1927-2017),
     Lewis Rowell, William Thomson,,
   Elisabeth K\"{u}bler-Ross (1926-2004),,
   Hui-Wen Teng
      for books:   	
   Gareth Green,,
   Jung-Hsuan Ko,,
   Oscar Osicki
      for educational series:
   creators of the HalloDeutschschule.ch for program
   that accompany the brewing and preparation of works from projects, fall 2021-spring 2022;
  Dan Forrest
      for the choral  piece {\sl Requiem for the Living} and
  Bob Jones University Chorale and
  Rivertree Singers \& Friends,
  both conducted by Warren Cook,
      for their performances that accompany the typing of the current notes;
  Pei-Jung Chen
      for the biweekly communications on a work of J.S.\ Bach (1685-1750)
	    that set another anchor for life in the COVID-19 times
%	    mixed with the world-wide uneasiness since another forever-world-changing event of
%	    the Russian invasion of Ukraine (2/24/2022)
      and
  Ling-Miao Chou
        for the daily exchange of progress,
		comments that improve the illustration,
		and the tremendous moral support.
% The project is supported by NSF grants DMS-9803347 and DMS-0074329.
} %endscriptsize

\end{titlepage}

\newpage

\enlargethispage{24cm}
\begin{titlepage}

$ $

%\vspace{2em}
\vspace{6em}

\centerline{\small\it
 Dedicated to those who fought or are fighting on the front lines of the COVID-19 pandemic.}
% \centerline{\small\it
%  ??????????????}

% %\bigskip
% \vspace{12em}
%
% \noindent
% {\footnotesize
% {\it From C.-H.L.}:\hspace{1em}
%  ?????????????.
% } % end-footnotesize/scriptsize
 
 % \vspace{5em}
 %
 % \baselineskip 11pt
 %
 % \noindent
 % $^{\ast}${\scriptsize %\footnotesize
 % (From C.-H.L.)\hspace{1em}
 % ????????????????.
 % } % end-footnotesize/scriptsize

\end{titlepage}

%paper

\newpage
$ $

\vspace{-3em}
% \vspace{-4em}  % Re: -4cm for PC; -6cm for UT-Math-system

%short heading
\centerline{\sc
 Soft Noncommutative Flag Schemes}

\vspace{1.2em}

% \baselineskip 14pt  %Re: 14pt for [11pt] style
                                      %Re: 15pt for [12pt] style.

\begin{flushleft}
{\Large\bf 0. Introduction and outline}
\end{flushleft}
Beyond toric varieties, another class of varieties that admit versatile use in both mathematics and physics ---
  including Mirror Symmetry ---
 are flag varieties.
Indeed, except that the labelling of distinguished charts cannot be realized by a fan,
 their construction from gluing a finite collection of affine varieties,
          each isomorphic to an affine space $\mathbf{A}^d$ for some common $d$,
   shares similar features of toric varieties.
From this perspective, it is very natural to expect that
  the construction of {\it soft noncommutative schemes} via toric geometry in
     [L-Y2] (arXiv:2108.05328 [math.AG] (D(15.1), NCS(1)))
  can be extended to flag varieties as well.
The details are carried out in this work.
Similar to soft noncommutative toric schemes,
 this gives a version of `noncommutative flag schemes'
   that may serve as target-spaces for dynamical D-branes in String Theory, cf. [L-Y2: Sec.\ 4].

As a side remark, it should be noted that
 noncommutatization or quantization of flag manifolds or varieties has been a topic of interest for long, e.g.\ [M-S],
 though there is no obvious connection between soft noncommutative flag schemes constructed in the current work
 and any previous version of `noncommutative flag manifolds or varieties' that we know of.

\bigskip
%\bigskip

\noindent
{\bf Convention.}
 References for standard notations, terminology, operations and facts are\\
  (1) toric geometry: [Fu2];  \hfill  %\hspace{.6em}
  (2) aspects of noncommutative algebraic geometry: [B-R-S-S-W];\\
  (3) (commutative) algebraic geometry: [E-H], [Ha];  \\  %\hspace{.6em}
  (4) Grassmann or flag manifolds or varieties: [B-T], [Fu1], [G-H].
 \begin{itemize}
  %-------------------------------------------------------------
  % \item[$\cdot$]
  %   For clarity, the {\it real line} as a real $1$-dimensional manifold is denoted by ${\Bbb R}^1$,
  %   while the {\it field of real numbers} is denoted by ${\Bbb R}$.
  %  Similarly, the {\it complex line} as a complex $1$-dimensional manifold is denoted by ${\Bbb C}^1$,
  %   while the {\it field of complex numbers} is denoted by ${\Bbb C}$.
  %	
  %  \item[$\cdot$]	
  %  The inclusion `${\Bbb R}\subset{\Bbb C}$' is referred to the {\it field extension
  %   of ${\Bbb R}$ to ${\Bbb C}$} by adding $\sqrt{-1}$, unless otherwise noted.
  %
  % \item[$\cdot$]	
  %   The {\it real $n$-dimensional vector spaces} ${\Bbb R}^{\oplus n}$
  %       vs.\ the {\it real $n$-manifold} $\,{\Bbb R}^n$; \\
  %   similarly, the {\it complex $r$-dimensional vector space ${\Bbb C}^{\oplus r}$}
  %      vs.\ the {\it complex $r$-fold} $\,{\Bbb C}^r$.
  %
  % \item[$\cdot$]
  %  All manifolds are paracompact, Hausdorff, and admitting a (locally finite) partition of unity.
  %  We adopt the {\it index convention for tensors} from differential geometry.
  %   In particular, the tuple coordinate functions on an $n$-manifold is denoted by, for example,
  %   $(y^1,\,\cdots\,y^n)$.
  %  However, no up-low index summation convention is used.
  %==-------------------------------------
  
  \item[$\cdot$]
  All commutative schemes are over ${\Bbb C}$ and Noetherian.
  
  \item[$\cdot$]
  {\it Index-related sets} $I$, $I_j$, $\mathbf{I}$, $\underlineboldI$, $\underlineboldII$ 
     from the set $\{1, \cdots, n\}$
   vs.\ {\it ideal}  $I$ of a ring.
  
  % \item[$\cdot$]
  % ??????????
  % All varieties, schemes and their products are over ${\Bbb C}$;
  % a `{\it curve}' means a $1$-dimensional proper scheme over ${\Bbb C}$.
  % % a `{\it stack}' means an {\it Artin stack}.
  
  \item[$\cdot$]
  Basic terminology of soft noncommutative schemes follows [L-Y2] when applicable.

 \end{itemize}

\bigskip
%\bigskip
%\newpage
   
\begin{flushleft}
{\bf Outline}
\end{flushleft}
\nopagebreak
{\small
\baselineskip 12pt  %13pt
\begin{itemize}
 \item[0]
  Introduction      \hfill 1

 \item[1]
  \makebox[42.8em][s]{Flag varieties and their distinguished atlas     \hfill 2}
  \vspace{-3.4ex}
  \begin{itemize}	 	
    \item[$\LARGEdot$]	
    A distinguished atlas on a flag variety
	  \dotfill \makebox[2em][r]{2}

    \item[$\LARGEdot$]
	The reference chart $U_{\mathbf{I}_0}\in {\cal U}_0$ and its subordinates
	  \dotfill \makebox[2em][r]{3}
	
    \item[$\LARGEdot$]
	A general chart $U_{\mathbf{I}}\in{\cal U}_0$ and its subordinates
	  \dotfill \makebox[2em][r]{5}
	
    \item[$\LARGEdot$]
	Realization of all $R_\tinybullet$ in the master ring
	  $R_{\underline{\mathbf{I}\!\mathbf{I}}}$
	  for {\it Fl}\,$(d_1, \cdots, d_r; n)$
	   \dotfill \makebox[2em][r]{6}		
  \end{itemize}
  
 \item[2]
  \makebox[42.8em][s]{Soft noncommutative flag schemes and their construction  \hfill 7}
  \vspace{-3.4ex}
  \begin{itemize}	 	
    \item[$\LARGEdot$]	
    Soft noncommutative flag schemes
	   \dotfill \makebox[2em][r]{7}
   
	\item[$\LARGEdot$]
	Construction of soft noncommutative flag schemes
	   \dotfill \makebox[2em][r]{8}

    \item[$\LARGEdot$]
	Closed subschemes of $\breve{X}_\underlineboldII$
	   \dotfill \makebox[2em][r]{9}	
	
    \item[$\LARGEdot$]
    Example: {\it Gr}\,$(2; 4)$
	   \dotfill \makebox[2em][r]{10}
	
	\item[$\LARGEdot$]
	Soft noncommutative Calabi-Yau schemes in $\breve{X}_\underlineboldII$ associated to {\it Gr}\,$(2; 4)$?	
	   \dotfill \makebox[2em][r]{11}
  \end{itemize}
  
 % \item[?]
 %  \makebox[42.8em][s]{??????????????     \hfill ??}
 %  %
 %  \vspace{-3.6ex}
 %  \begin{itemize}	 	
 %    \item[?.1]	
 %	??????????????????????
 %	\dotfill \makebox[2em][r]{??}
 %	
 %	\item[?.2]
 %	??????????????????????
 %	\dotfill \makebox[2em][r]{??}
 %  \end{itemize}
 %
 %  \item[?]
 %  ?????????????????????
 %  %
 %  \vspace{-.6ex}
 %  \begin{itemize}	 	
 %    \item[?.1]	
 %	
 %	\item[?.2]
 %	 ?????????????????
 %	 %
 %	  \begin{itemize}
 %	   \item[?.2.1]
 %        ??????????????????????
 %	
 %	   \item[?.2.2]
 %		?????????????????????
 %		%
 %		\begin{itemize}
 %		  \item[?.2.2.1]
 %            ???????????????????????
 %
 %          \item[?.2.2.2]
 %            ???????????????????????
 %		\end{itemize}
 %	  \end{itemize}
 %  \end{itemize}
\end{itemize}
} %endsmall

\newpage

\section{Flag varieties and their distinguished atlas}

A contravariant construction of a flag variety in terms of a distinguished gluing system of rings
 is presented in this section.
All the material here is standard
 but the presentation gives an immediate comparison of flag varieties to toric varieties.

\bigskip

\begin{flushleft}
{\bf A distinguished atlas on a flag variety}
\end{flushleft}
Let  ${\Bbb C}^n$ be the $n$-dimensional vector space over $\Bbb C$.
For integers $0 < d_1< d_2< \cdots < d_r < n$,
 let $\Fl(d_1,\cdots, d_r; n)$ denote the flag variety
   whose ${\Bbb C}$-points are flags of vector subspaces
    $$
	   L_1\subset L_2\subset \cdots \subset L_r \subset {\Bbb C}^n
	$$
   with $\dim L_i=d_i$.
Fix an isomorphism ${\Bbb C}^n\simeq {\Bbb  C}^{\oplus n}$ and
 the standard basis $e_1,\,\cdots\,,e_n$ under the isomorphism,
 with $e_i:= (0,\cdots, 0, 1, 0,\cdots, 0)$ where $1$ is in the $i$-th entry.
For a nonempty  subset $I\subset \{1,\cdots,n\}$,
  let $H_I$ be the vector subspace $\Span_{\Bbb C}\{e_i | i\in I\}$ of ${\Bbb C}^n$.
The canonical decomposition ${\Bbb C}^n=H_I\oplus H_{I^c}$, where $I^c:= \{1,\cdots, n\}-I$,
 defines a projection map $\pi_I: {\Bbb C}^n\rightarrow H_I$.
For an inclusion sequence $\mathbf{I}$ of subsets
    $$
      I_1\subset I_2 \subset \cdots \subset I_r \subset \{1,\cdots, n\}
    $$
    with cardinality $|I_i|=d_i$,
  let $U_\mathbf{I}$ be the open affine subvariety of $\Fl(d_1,\cdots, d_r; n)$
   whose ${\Bbb C}$-points are flags
    $$
	   L_1\subset L_2\subset \cdots \subset L_r \subset {\Bbb C}^n
	$$
	such that the restriction $\pi_{I_i}: L_i \rightarrow H_{I_i}$
	  is a ${\Bbb C}$-vector-space  isomorphism for $i=1,\cdots, r$.
As a variety each $U_\mathbf{I}$ is isomorphic to the affine space
 $\mathbf{A}^{d_r(n-d_r)+ d_{r-1}(d_r-d_{r-1})+\cdots + d_1(d_2-d_1)}$.
For $\underline{\mathbf{I}}=\{\mathbf {I}_1, \cdots, \mathbf {I}_l\}$,
  where $\mathbf{I}_j$ are as above,
 let
 $$
  U_{\underline{\mathbf{I}}}  \; :=\;  \bigcap_{j=1}^l U_{\mathbf{I}_j}\,.
 $$
Then, though no longer an affine space in general,
 $U_{\underline{\mathbf{I}}}$ remains affine and the collection satisfies
 $$
    U_{\underline{\mathbf{I}}}\cap U_{\underline{\mathbf{J}}}\;
	  =\;  U_{\underline{\mathbf{I}}\,\cup\, \underline{\mathbf{J}}}\,.
 $$

\medskip

\begin{sdefinition} {\bf [admissible sequence and admissible chain]}\; {\rm
 An inclusion sequence $\mathbf{I}$ of subsets
   $I_1\subset I_2 \subset \cdots \subset I_r \subset \{1,\cdots, n\}$ with cardinality $|I_i|=d_i$
  is called an {\it admissible sequence}  of subsets of $\{1,\cdots,n\}$.
 A collection $\underline{\mathbf{I}}=\{\mathbf{I}_1,\cdots,\mathbf{I}_l\}$
   of distinct admissible sequences of subsets of $\{1,\cdots,n\}$
   is a called an {\it admissible chain} from $\{1,\cdots,n\}$.
}\end{sdefinition}

\medskip

Note that
 $\{U_\mathbf{I}\,|\,\mbox{$\mathbf{I}$: admissible sequence}\}$ is an affine open cover
   of $\Fl(d_1, \cdots, d_r; n)$
 and that,
  for $\mathbf I: I_1\subset \cdots \subset I_r\subset \{1, \cdots, n\}$ an admissible sequence,
  the flag $H_{I_1}\subset \cdots \subset H_{I_r}\subset {\Bbb C}^n$ from nested coordinate planes
  lies only in $U_\mathbf{I}$.
Thus, this is a  minimal cover in the sense that the collection in this cover cannot be reduced to maintain a cover.
  
\medskip

\begin{sdefinition} {\bf [distinguished atlas on $\Fl(d_1, \cdots, d_r; n)$]}\; {\rm
 By construction the finite system
 ${\cal U}_0 := \{U_{\underline{\mathbf{I}}}\}_{\underline{\mathbf{I}}}$
  of affine open subsets $U_{\underline{\mathbf{I}}}$ of $\Fl(d_1, \cdots, d_r; n)$,
     where $\underline{\mathbf{I}}$ runs over all admissible chains from $\{1,\cdots,n\}$,
  covers $\Fl(d_1, \cdots, d_r; n)$ and is closed under taking intersections;
 it is called the {\it distinguished atlas} on $\Fl(d_1,\cdots, d_r; n)$
  associated to the ${\Bbb C}$-vector-space isomorphism ${\Bbb C}^n\simeq {\Bbb C}^{\oplus n}$.
 For $U_{\underline{\mathbf{I}}}, U_{\underline{\mathbf{J}}}\in {\cal U}_0$,
  we say that
   {\it $U_{\underline{\mathbf{I}}}$ is subordinate to $U_{\underline{\mathbf{J}}}$}
    or interchangeably that
   {\it $U_{\underline{\mathbf{I}}}$ is a subordinate of $U_{\underline{\mathbf{J}}}$}
  if $U_{\underline{\mathbf{I}}}\subset U_{\underline{\mathbf{J}}}$.
 A {\it maximal chart} in ${\cal U}_0$ with respect to this partial order $\subset$  is exactly
  $U_{\mathbf{I}}$ for some admissible sequence $\mathbf{I}$.
}\end{sdefinition}

\medskip

\begin{sremark} $[$weight polytope$]$\; {\rm
 When $\Fl(d_1, \cdots, d_r; n)$ is realized as the orbit of a highest weight vector
   in the projectivization $\mathbf{P}V$ of a complex representation $V$ of the special unitary group $\SU(n)$
   of rank $n-1$,
 the set
   $\underline{\mathbf{I}\!\mathbf{I}}
     :=\{\underline{\mathbf{I}}\,|\,\mbox{$\underline{\mathbf{I}}$ is an admissible sequence}\}$
   that labels maximal charts in ${\cal U}_0$ can be identified with the set  of vertices of the weight polytope
   in the dual Carton subalgebra of $\SU(n)$ associated to the representation.
 (Cf.\ [F-H: Sec.\ 23.3], [M-S, Sec.\ 3].)
}\end{sremark}

\medskip

\noindent
We shall now describe this distinguished atlas contravariantly in terms of local coordinate rings and their localizations.

\bigskip

\begin{flushleft}
{\bf The reference chart $U_{\mathbf{I}_0}\in {\cal U}_0$ and its subordinates}
\end{flushleft}
Let $\mathbf{I}_0$ be the admissible sequence
 $$
   \{1, \cdots, d_1\}\; \subset\; \{1, \cdots, d_2\} \; \subset\; \cdots\;
    \subset\; \{1, \cdots, d_r \}\; \subset\; \{1, \cdots, n \}\,.
 $$
Then the standard coordinate-ring associated to the chart $U_{\mathbf{I}_0}$
 is given by the polynomial ring over ${\Bbb C}$
 $$
  R_{\mathbf{I}_0}\; :=\;
  {\Bbb C}\left[
       z_{ij}\left|
	     \begin{array}{l}
	       \mbox{$j\in \{d_k+1,\cdots, n\}$ for $i\in \{d_{k-1}+1,\cdots, d_k\}$}, \\
          \mbox{$k=1,\cdots, r$ with $d_0 = 0$ by convention}		
		 \end{array}
		            \right. \hspace{-1ex}\right]
 $$
 from the blocked matrix presentation of a flag
    $\mathbf{L}:\;L_1\subset L_2\subset \cdots \subset L_r \subset {\Bbb C}^n$
    that corresponds to a ${\Bbb C}$-point on $U_{\mathbf{I}_0}$
 $$
  M_{\mathbf{I}_0}(\mathbf{L})\;
  :=\;
  \left[
   \begin{array}{l}
    \colorbox{gray07}{\hbox to 2ex{\vbox to 2ex{}}}\hspace{.6ex}
       \colorbox{gray25}{\hbox to 4ex{\vbox to 2ex{}}}\hspace{.6ex}
	   \colorbox{gray25}{\hbox to 3ex{\vbox to 2ex{}}}\hspace{.6ex}	      		
	   \colorbox{gray25}{\hbox to 7ex{\vbox to 2ex{}}}\hspace{.6ex}
	   \colorbox{gray25}{\hbox to 12ex{\vbox to 2ex{}}}\hspace{.6ex}	    	
		    \\[.4ex]
	\colorbox{gray02}{\hbox to 2ex{\vbox to 4ex{}}}\hspace{.6ex}
	   \colorbox{gray07}{\hbox to 4ex{\vbox to 4ex{}}}\hspace{.6ex}
       \colorbox{gray25}{\hbox to 3ex{\vbox to 4ex{}}}\hspace{.6ex}
	   \colorbox{gray25}{\hbox to 7ex{\vbox to 4ex{}}}\hspace{.6ex}
	   \colorbox{gray25}{\hbox to 12ex{\vbox to 4ex{}}}
		   \\[.4ex]
   	\colorbox{gray02}{\hbox to 2ex{\vbox to 3ex{}}}\hspace{.6ex}
	   \colorbox{gray02}{\hbox to 4ex{\vbox to 3ex{}}}\hspace{.6ex}	
       \colorbox{gray07}{\hbox to 3ex{\vbox to 3ex{}}}\hspace{.6ex}
       \colorbox{gray25}{\hbox to 7ex{\vbox to 3ex{}}}\hspace{.6ex}
	   \colorbox{gray25}{\hbox to 12ex{\vbox to 3ex{}}}
		   \\[.4ex]		
	\colorbox{gray02}{\hbox to 2ex{\vbox to 7ex{}}}\hspace{.6ex}
	   \colorbox{gray02}{\hbox to 4ex{\vbox to 7ex{}}}\hspace{.6ex}	
       \colorbox{gray02}{\hbox to 3ex{\vbox to 7ex{}}}\hspace{.6ex}
       \colorbox{gray07}{\hbox to 7ex{\vbox to 7ex{}}}\hspace{.6ex} 		
	   \colorbox{gray25}{\hbox to 12ex{\vbox to 7ex{}}}
   \end{array}
  \right]_{d_r\times n}
 $$
 where
  \begin{itemize}
   \item[$\LARGEdot$]
   the diagonal blocks (in    \colorbox{gray07}{\hbox to 1ex{\vbox to 1ex{}}})
   from upper-left to lower-right are respectively the identity matrices
   $\Id_{d_1\times d_1},\, \Id_{(d_2-d_1)\times(d_2-d_1) },\,\cdots,
     \Id_{(n-d_r)\times(n-d_r)}   $;

   \item[$\LARGEdot$]	
   all the blocks (in    \colorbox{gray02}{\hbox to 1ex{\vbox to 1ex{}}}) below the diagonal blocks
   are zero;
   
   \item[$\LARGEdot$]
   for $i=1,\cdots, r$, $L_i$ is the span of the upper $d_i$-many row-vectors of
    $M_{\mathbf{I}_0}(\mathbf{L})$.
  \end{itemize}
When $\mathbf{L}$ varies in $U_{\mathbf{I}_0}$,
 the $(i,j)$-entry of $M_{\mathbf{I}_0}(\mathbf{L})$
   that sits in a block (in    \colorbox{gray25}{\hbox to 1ex{\vbox to 1ex{}}})
	 to the right of a diagonal block gives rise to the coordinate function $z_{ij}$
	  in the coordinate-ring $R_{\mathbf{I}_0}$ of $U_{\mathbf{I}_0}$.
As a bookkeeping device for later discussions, denote
 \begin{eqnarray*}
  M_{\mathbf{I}_0}(\boldz)
    & := & \mbox{the above $d_r\times n$ matrix
              with every $(i,j)$-entry in the \colorbox{gray25}{\hbox to 1ex{\vbox to 1ex{}}}\,-region}
			  \\[-.6ex]
    && 	\mbox{(i.e. blocks to the right of diagonal blocks) replaced by its associate $z_{ij}$}
	          \\[-.6ex]
	&&   \mbox{while keeping all the digonal blocks
	                         (cf.\ the \colorbox{gray07}{\hbox to 1ex{\vbox to 1ex{}}}\,-region)}
			  \\[-.6ex]
    &&   \mbox{and the zero-blocks
						     (cf.\ the \colorbox{gray02}{\hbox to 1ex{\vbox to 1ex{}}}\,-region).}
 \end{eqnarray*}
Thus, $R_{\mathbf{I}_0}$ is the polynomial ring over ${\Bbb C}$ generated by entries of
 $M_{\mathbf{I}_0}(\boldz)$.
In notation,
 $$
   R_{\mathbf{I}_0}\;=\; {\Bbb C}[M_{\mathbf{I}_0}(\boldz)].
 $$
  
For a general admissible  sequence
  $\mathbf{I}:\: I_1 \subset I_2 \subset \cdots \subset I_r \subset \{1,\cdots, n\}$ with	
 $$
  \begin{array}{rclccl}
   I_1  & =  & \{i_{1,1},\cdots, i_{1,d_1}\},
           &&&  i_{1,1} < \cdots < i_{1,d_1};
           \\[1.2ex]
   I_2  & = & \{i_{1,1},\cdots, i_{1,d_1}, i_{2,1}, \cdots, i_{2, d_2-d_1}\},
           &&& i_{2,1} < \cdots < i_{2, d_2-d_1};
           \\[1.2ex]
     & \vdots & \\[1.2ex]
   I_r  & = & \{i_{1,1},\cdots, i_{1,d_1}, i_{2,1}, \cdots, i_{2, d_2-d_1}, \cdots,
                       i_{r,1}, \cdots, i_{r, d_r-d_{r-1}}\},
		   &&&   i_{r,1} < \cdots < i_{r, d_r-d_{r-1}}
  \end{array}
 $$
 and an $ r\times n$ matrix $M$,
let us introduce two sets of convenient bookkeeping notations as follows.

\smallskip

\noindent
(1) The {\it distinguished square submatrices}
  $$
    M^{I_1}, M^{I_2}, \cdots, M^{I_r}  \hspace{2em}\mbox{and}\hspace{2em}
    M^{I_2-I_1}, M^{I_3-I_2}, \cdots, M^{I_r-I_{r-1}}
  $$
  of $M$, where
 $$
  \begin{array}{rcl}
   M^{I_j}  & :=
     & \mbox{the $d_j\times d_j$ submatrix of $M$
	     by removing all the $j^\prime$-th rows and} \\[.2ex]
     &&	\mbox{the $j^{\prime\prime}$-th columns of $M$
		 with $j^\prime>d_j$ and $j^{\prime\prime}\notin I_j$},
		\\[1.6ex]
   M^{I_j-I_{j-1}}  & :=
     & \mbox{the $(d_j-d_{j-1})\times (d_j-d_{j-1})$ submatrix of $M$
	     by removing all the $j^\prime$-th rows} \\[.2ex]
     &&	\mbox{and the $j^{\prime\prime}$-th columns of $M$
		 with $j^\prime \le d_{j-1}$ or $j^\prime>d_j$
		        and $j^{\prime\prime}\notin I_j-I_{j-1}$}.		
  \end{array}
 $$
Note that, by construction, $M^{I_j-I_{j-1}}$ is a square submatrix of $M^{I_j}$ as well   and
  that the collection $M^{I_1}, M^{I_2-I_1}, M^{I_3-I_2}, \cdots, M^{I_r-I_{r-1}}$ are
    non-overlapping in $M$.

\smallskip

\noindent
(2) The {\it characteristic map} $\chi_\mathbf{I}: \{1, \cdots, n\}\rightarrow \{1, \cdots, n\}$,
          $k\mapsto  i_{j, k-d_{j-1}}$ if $k \in \{d_{j-1}+1, \cdots, d_j\}$, $j=1, \cdots, r, r+1$.
 This is the unique set-automorphism of $\{1, \cdots, n\}$ that takes $\{1, \cdots, d_j\}$ to $I_j$
  such that the restriction $\{d_{j-1}+1, \cdots, d_j\}\rightarrow I_j-I_{j-1}$ is order-preserving
    for all $j=1, \cdots, r, r+1$.
 Here
   $d_0=0$,
   $d_{r+1}=n$,
   $I_{r+1}= \{1, \cdots, n\}$,
   $I_{r+1}-I_r=  \{i_{r+1, 1}, \cdots, i_{r+1, n-d_r}\}$
     with $i_{r+1, 1}< \cdots < i_{r+1, n-d_r}$
  by convention.
 Associated to $\chi_\mathbf{I}$ is the {\it characteristic $n\times n$ matrix} $C_{\chi_\mathbf{I}}$
  such that $MC_{\chi_\mathbf{I}} = $ the $r\times n$ matrix obtained from $M$
   by exchanging the $k$-th column with $\chi_\mathbf{I}(k)$-th column, $k=1, \cdots, n$.
 Note that $C_{\chi_\mathbf{I}}^{\;2}= \Id_{n\times n}$.	

\smallskip
	
Let $\mathbf{I}$ now be an admissible sequence of subsets of $\{1,\cdots,n\}$
   that is distinct from $\mathbf{I}_0$    and
consider the affine open subset
   $U_{\{\mathbf{I}_0, \mathbf{I}\}}:= U_{\mathbf{I}_0}\cap U_{\mathbf{I}}$
   of $\Fl(d_1, \cdots, d_r; n )$
 associated to the chain $\{\mathbf{I}_0, \mathbf {I}\}$.
A flag $\mathbf{L}\in U_{\mathbf{I}_0}$ lies also in $U_{\mathbf{I}}$
 if and only if the restriction of the projection map $\pi_{I_j}: L_j\rightarrow H_{I_j}$ is an isomorphism
  for $I_j\in\mathbf{I}$, $j=1,\cdots, r$.
The latter is true if and only if
 all the square submatrices
   $$
     M_{\mathbf{I}_0}(\mathbf{L})^{I_1},\,
	 M_{\mathbf{I}_0}(\mathbf{L})^{I_2},\,\cdots,\,
	 M_{\mathbf{I}_0}(\mathbf{L})^{I_r}
   $$
   of $M_{\mathbf{I}_0}(\mathbf{L})$ are invertible, which is equivalent to that
 all the square submatrices
  $$
    M_{\mathbf{I}_0}(\mathbf{L})^{I_1},\,
	M_{\mathbf{I}_0}(\mathbf{L})^{I_2-I_1},\,\cdots,\,
	M_{\mathbf{I}_0}(\mathbf{L})^{I_r-I_{r-1}}
  $$
  of $M_{\mathbf{I}_0}(\mathbf{L})$ are invertible.
It follows that
 \begin{itemize}
  \item[$\LARGEdot$]
   The function-ring $R_{\{\mathbf{I_0}, \mathbf{I}\}}$
      of $U_{\{\mathbf{I_0}, \mathbf{I}\}}$
	  is given by the localization of $R_{\mathbf{I}_0}$ at the multiplicatively closed subset generated by
	 $$
	   S_{\mathbf{I}_0; \mathbf{I}}\;:=\;
	    \{ \determinant M_{\mathbf{I}_0}(\boldz)^{I_1},\,
	         \determinant M_{\mathbf{I}_0}(\boldz)^{I_2-I_1},\,\cdots,\,
	         \determinant M_{\mathbf{I}_0}(\boldz)^{I_r-I_{r-1}}		
		\},
	 $$
	 where $\determinant \bullet$ is the determinant of the square matrix $\bullet$\,.
 \end{itemize}
With slight abuse of notation, we shall denote this localized ring as
     $R_{\mathbf{I}_0}[S_{\mathbf{I}_0; \mathbf{I}}^{\;\;-1}]$.
Hence
 $$
  R_{\{\mathbf{I_0}, \mathbf{I}\}}\;
   \simeq\;     R_{\mathbf{I}_0}[S_{\mathbf{I}_0; \mathbf{I}}^{\;-1}]\;
   \simeq\;
      \frac{{\Bbb C}[z_{ij}, w_k\, |\,
                                   \mbox{position-$(i,j)$ in blocks
                                      in \colorbox{gray25}{\hbox to 1ex{\vbox to 1ex{}}}\,;
								      $k=1, \cdots, r$
								    } ]}
               {(w_k\cdot \determinant M_{\mathbf{I}_0}(\boldz)^{I_k-I_{k-1}} -1
			          \,|\, k=1, \cdots, r)\rule{0ex}{2.2ex}}\,.									
 $$
Here,
  the numerator is the polynomial ring generated by the variables $z_{ij}$'s and $w_k$'s as specified,
  the denominator is the ideal of the numerator generated by the elements as indicated,  and
  $I_0$ is the empty set by convention.
  
The same discussion applies to the affine open subset $U_{\underline{\mathbf{I}}}$ of
 $U_{\mathbf{I}_0}\subset \Fl(d_1, \cdots, d_r; n)$ associated to a higher chain
 $\underline{\mathbf{I}} =\{\mathbf{I}_0, \mathbf{I}_1, \cdots, \mathbf{I}_l  \}$;
which gives:
 \begin{itemize}
  \item[$\LARGEdot$]
   The function-ring $R_{\underline{\mathbf{I}}}$
      of $U_{\underline{\mathbf{I}}}$
	  is given by the localization of $R_{\mathbf{I}_0}$ at the multiplicatively closed subset generated by
	 $$
	   S_{\mathbf{I}_0; \{\mathbf{I}_1, \cdots, \mathbf{I}_l\}   }\;:=\;
	    \{  \determinant M_{\mathbf{I}_0}(\boldz)^{I_{i, j}-I_{i, j-1}}		
		        \,|\, i=1, \cdots, l;\, j=1, \cdots, r
		  \} ,
	 $$
   where $I_{i, \tinybullet}\in \mathbf{I}_i$ and $I_{\tinybullet, 0}$ are the empty set by convention.
   I.e.\ (with slight abuse of notation)
   $$
      R_{\underline{\mathbf{I}}}\;
       \simeq\;
	     R_{\mathbf{I}_0}
	     [S_{\mathbf{I}_0; \{\mathbf{I}_1, \cdots, \mathbf{I}_l\} }^{\;-1}] \; \
		 =\; {\Bbb C}[ M_{\mathbf{I}_0}(\boldz)]
		                            [S_{\mathbf{I}_0; \{\mathbf{I}_1, \cdots, \mathbf{I}_l\}}^{\;-1}]\;
		 =\; {\Bbb C}[ M_{\mathbf{I}_0}(\boldz)]
		                            [S_{\mathbf{I}_0; \{\mathbf{I}_1, \cdots, \mathbf{I}_l\}, \circ}^{\;-1}].
   $$
 \end{itemize}
 Here,
   notice that there could be repeating elements in the set
     $S_{\mathbf{I}_0; \{\mathbf{I}_1, \cdots, \mathbf{I}_l\}   }$,
       which are redundant as long as localizations of $R_{\mathbf{I}_0}$ are concerned,   and
   $S_{\mathbf{I}_0; \{\mathbf{I}_1, \cdots, \mathbf{I}_l\}, \circ}$
    is the pruned $S_{\mathbf{I}_0; \{\mathbf{I}_1, \cdots, \mathbf{I}_l\}   }$	
   that removes all such repeating elements.
In particular,

\medskip

\begin{sdefinition} {\bf [master ring for $\Fl(d_1, \cdots, d_r; n)$]}\; {\rm
  Recall the set
    $\underline{\mathbf{I\!I}}=\{\mathbf{I}_0\}\amalg \underline{\mathbf{I\!I}}^\prime$
    of all admissible sequences of subsets of $\{1, \cdots, n\}$.
  Note that $\underline{\mathbf{I\!I}}$ has cardinality
     $$
	   {n \choose d_r} {d_r \choose d_{r-1}} \cdots {d_2 \choose d_1}\;
	    =\;  \frac{n!}{(n-d_r)! \,(d_r-d_{r-1})!\, \cdots\, (d_2-d_1)!\,  d_1!   \rule{0ex}{2ex}}
	 $$
	 as a set of admissible sequences
    and is the the maximal admissible chain.
  We shall call
    $$
      R_{\underline{\mathbf{I\!I}}}\;   \simeq\;
	  R_{\mathbf{I}_0} [S_{\mathbf{I_0}; \underline{\mathbf{I\!I}}^\prime, \circ}^{\;-1}]\;
      =\; {\Bbb C}[M_{\mathbf{I}_0}(\boldz)]
	                              [S_{\mathbf{I_0}; \underline{\mathbf{I\!I}}^\prime, \circ}^{\;-1}]\; 								  
    $$	
  the {\it master ring} for the flag variety $\Fl(d_1, \cdots, d_r; n)$.
 This is the function-ring of the smallest affine open subset
   $\bigcap_{\mathbf{I}\in \underline{\mathbf{I}\!\mathbf{I}}} U_\mathbf{I}$
   in the distinguished atlas ${\cal U}_0$ on $\Fl(d_1, \cdots, d_r; n)$.
}\end{sdefinition}

\medskip

\begin{sdefinition} {\bf [distinguished submatrix]}\; {\rm
 For a $d_r\times n$  matrix $M$ and $i_1, i_2\in \{1, \cdots, d_r\}$ with $i_1\le i_2$,
    denote by $_{[i_1, i_2]}M$
    the $(i_2-i_1+1)\times n$ submatrix of $M$ by removing all $i$-th rows with $i <i_1$ or $i>i_2$.
 Then explicitly,
    $S_{\mathbf{I}_0; \underline{\mathbf{I}\!\mathbf{I}}^\prime,\, \circ}$
	 is the set of the determinant of all $(d_i-d_{i-1})\times (d_i-d_{i-1})$ submatrices of
	$_{[d_{i-1}+1, d_i]}M_{\mathbf{I}_0}(\boldz)$, $i=1, \cdots, r$, $d_0=0$ by convention.
 For convenience, we shall call any of the latter square submatrices
   a {\it distinguished submatrix} of $M_{\mathbf{I}_0}(\boldz)$.
}\end{sdefinition}

\bigskip

\begin{flushleft}
{\bf A general chart $U_{\mathbf{I}}\in {\cal U}_0$ and its subordinates}
\end{flushleft}
For a general admissible sequence $\mathbf{I}$,
 recall the characteristic map $\chi_\mathbf{I}$ and the characteristic matrix $C_{\chi_\mathbf{I}}$.
Then, similar to the function-ring $R_{\mathbf{I}_0}$ of $U_0$,
 the function-ring $R_\mathbf{I}$ of $U_\mathbf{I}$ is given by
 the polynomial ring over ${\Bbb C}$ generated by the entries of
  $$
    M_\mathbf{I}(\boldx)\;:=\; M_{\mathbf{I}_0}(\boldx) C_{\chi_I}\,.
  $$
Let $\mathbf{J}$ be another admissible sequence   and
  $R_\mathbf{J}$ be the polynomial-ring over ${\Bbb C}$ generated by the entries of
  $M_\mathbf{J}(\boldy):= M_{\mathbf{I}_0}(\boldy) C_{\chi_\mathbf{J}}$,
  which gives the function-ring of $U_\mathbf{J}$.
Then,
  the function-ring of $U_{\{\mathbf{I}, \mathbf{J}\}}:= U_\mathbf{I}\cap U_\mathbf{J}$
  is given by
  $$
   R_{\{\mathbf{I}, \mathbf{J}\}}\;
   \simeq\; R_\mathbf{I}[S_{\mathbf{I}; \mathbf{J}}^{\;-1}]\;
   \simeq\; R_\mathbf{J}[S_{\mathbf{J};\mathbf{I}}^{\;-1}]\,,
  $$
  where
   \begin{eqnarray*}
    S_{\mathbf{I}; \mathbf{J}}\;
	 & := &  \{ \determinant M_\mathbf{I}(\boldx)^{J_1},\,
	                  \determinant M_\mathbf{I}(\boldx)^{J_2-J_1},\,\cdots,\,
	                  \determinant M_\mathbf{I}(\boldx)^{J_r-J_{r-1}}		
					 \},  \\[1.2ex]
    S_{\mathbf{J}; \mathbf{I}}\;
	 & := &  \{ \determinant M_\mathbf{J}(\boldy)^{I_1},\,
	                  \determinant M_\mathbf{J}(\boldy)^{I_2-I_1},\,\cdots,\,
	                  \determinant M_\mathbf{J}(\boldy)^{I_r-I_{r-1}}		
					 \},		
	 \end{eqnarray*}
   and the isomorphism
    $$
	 \phi^\sharp_{\mathbf{I}\mathbf{J}}\; :\;
	  R_\mathbf{J}[S_{\mathbf{J}; \mathbf{I}}^{\;-1}]\;
	  \stackrel{\sim}{\longrightarrow}\;
	  R_\mathbf{I}[S_{\mathbf{I}; \mathbf{J}}^{\;-1}]
	$$
	is given by the entry-wise correspondence of matrices:	
	$$	
	  M_\mathbf{J}(\boldy)\;   \longrightarrow\;
	    C_{\mathbf{I}; \mathbf{J}} M_\mathbf{I}(\boldx) ,
	$$
	where 	
	$C_{\mathbf{I}; \mathbf{J}}$ is the unique $r \times r$ matrix with entries
	in $R_\mathbf{I}[S_{\mathbf{I}; \mathbf{J}}^{\;-1}]$
	such that
	$C_{\mathbf{I}; \mathbf{J}}M_\mathbf{I}(\boldx) C_{\chi_\mathbf{J}}$
	 resumes the blocked matrix form
	 $M_{\mathbf{I}_0}(\mathbf{L})$ for a flag $\mathbf{L}\in U_0$.
Completely analogously for a deeper chain $\{\mathbf{I}_1, \cdots, \mathbf{I}_{l+1}\}$,
 all the isomorphisms
 $$
  R_{\underline{\mathbf{I}}_1}
    [S_{\underline{\mathbf{I}}_1^\prime;\,
	          \underline{\mathbf{I}}_2^\prime}^{\;-1}]\;
  \stackrel{\sim}{\longrightarrow}\;
  R_{\underline{\mathbf{I}}_2}
	[S_{\underline{\mathbf{I}}_1^{\prime\prime};\,
	          \underline{\mathbf{I}}_2^{\prime\prime}}^{\;-1}]
 $$
 are explicitly constructible this way for any decompositions
 $\{\mathbf{I}_1, \cdots, \mathbf{I}_{l+1}\}
   = \underline{\mathbf{I}}_1^\prime \cup \underline{\mathbf{I}}_2^\prime
  = \underline{\mathbf{I}}_1^{\prime\prime}\cup \underline{\mathbf{I}}_2^{\prime\prime}$
 of $\{\mathbf{I}_1, \cdots, \mathbf{I}_{l+1}\}$ into disjoint unions.

\bigskip
 
\begin{flushleft}
{\bf Realization of all $R_{\underline{\bullet}}$ in the master ring
         $R_{\underline{\mathbf{I}\!\mathbf{I}}}$ for $\Fl(d_1, \cdots, d_r; n)$}
\end{flushleft}
The built-in ${\Bbb C}$-algebra-homomorphisms
 $$
   \xymatrix{	
    R_\mathbf{I}\: \ar@{^{(}->}[r]	
	  & R_\mathbf{I}[S_{\mathbf{I}; \mathbf{I}_0}^{\;-1}]
	     \ar[rr]^-{\phi^\sharp_{\mathbf{I}_0\mathbf{I}}}_-\sim
	  && R_{\mathbf{I}_0}[S_{\mathbf{I}_0; \mathbf{I}}^{\;-1}]\: \ar@{^{(}->}[r]	   	
	  & R_{\underline{\mathbf{I}\!\mathbf{I}}},
	}
 $$
   for all admissible sequence $\mathbf{I}$,
 allow one to identify all the function-rings $R_\mathbf{I}$,
  and hence all $R_{\underline{\mathbf{J}}}$,
    where $\underline{\mathbf{J}}$ is an admissible chain,
  from localizations of some $R_{\mathbf{I}}$ as well,
 as ${\Bbb C}$-subalgebras of
  the master ${\Bbb C}$-algebra $R_{\underline{\mathbf{I}\!\mathbf{I}}}$
  for $\Fl(d_1, \cdots, d_r; n)$.	
Indeed, from the discussions in the previous two themes and with the notations therein,
 we shall make the following explicit identifications:
 \begin{itemize}
  \item[$\LARGEdot$]
  For $\mathbf{I}$ an admissible sequence,
   $$
    \begin{array}{rcl}
     R_\mathbf{I}&  = &
	   \mbox{the ${\Bbb C}$-subalgebra of
	                   $R_{\mathbf{I}_0}[S_{\mathbf{I}_0;\mathbf{I}}^{\;-1}]
  					      \subset R_{\underline{\mathbf{I}\!\mathbf{I}}}$
                    generated by the entries of} \\[.2ex]  && \mbox{the $r\times n$ matrix
                    $C_{\mathbf{I_0}; \mathbf{I}} M_{\mathbf{I}_0}(\boldz)$.}
    \end{array}					
   $$
   
  \item[$\LARGEdot$]
  For $\underline{\mathbf{I}}=\{\mathbf{I}_1, \cdots, \mathbf{I}_l\}$
   an admissible chain,
  $$
   R_{\underline{\mathbf{I}}}\;
    =\; \mbox{the ${\Bbb C}$-subalgebra of $R_{\underline{\mathbf{I}\!\mathbf{I}}}$
	      generated by $R_{\mathbf{I}_1} \cup \cdots \cup R_{\mathbf{I}_l}$.}
  $$
 \end{itemize}
and take this as the starting point to construct soft noncommutative flag schemes.

\medskip

\begin{snotation}
{\bf [contravariant description of $\Fl(d_1, \cdots, d_r; n)$ as $2^\underlineboldII$-system of subrings]}\;
{\rm
  One thus has a contravariant description of the flag variety $\Fl(d_1, \cdots, d_r; n)$
   as a $2^\underlineboldII$-system of ${\Bbb C}$-subalgebras in $R_\underlineboldII$.
  We shall denote this system by
   $$
     X_\underlineboldII :=({\cal U}_0, {\cal O}_{{\cal U}_0}),
   $$
    where ${\cal O}_{{\cal U}_0}$ is the {\it structure sheaf} of $\Fl(d_1, \cdots, d_r; n)$
	 adapted to ${\cal U}_0$,
	which assigns the subring $R_\underlineboldI \subset R_\underlineboldII$
	        to the chart $U_\underlineboldI\in {\cal U}_0$, for $\underlineboldI \subset \underlineboldII$.
}\end{snotation}
 
\medskip

\begin{sremark} $[$comparison to toric variety$\,]$\; {\rm
 The master ring $R_{\underline{\mathbf{I}\!\mathbf{I}}}$ for
   a flag variety
   will play the role of ${\Bbb C}[z_1,\cdots, z_n, z_1^{-1}, \cdots, z_n^{-1}]$ for some $n$
   (i.e.\ the function ring of the $n$-torus
          ${\Bbb T}^n_{\Bbb C}:=  ({\Bbb C}^\times)^n$ over ${\Bbb C}$
             or equivalently the group algebra ${\Bbb C}[M]$ associated to a lattice $M$ of rank $n$ for some $n$)
	in the case of toric varieties.
 In particular, the distinguished gluing system of ${\Bbb C}$-algebras associated to a flag variety
  is simply a special system of ${\Bbb C}$-subalgebras of the master ${\Bbb C}$-algebra
     $R_\underlineboldII$ labelled by admissible chains
  just like the distinguished gluing system of function-rings associated of a toric variety $X_\Delta/{\Bbb C}$
   being a system of ${\Bbb C}$-subalgebras in
   ${\Bbb C}[z_1,\cdots, z_n, z_1^{-1}, \cdots, z_n^{-1}]$ labelled by cones in the fan $\Delta$.
}\end{sremark}

\bigskip

\section{Soft noncommutative flag schemes and their construction}

With the preparation in Sec.\ 1,
 we can now proceed in the same way as [L-Y2] (D(15.1), NCS(1)),
    where the toric case is studied,
 to construct
  `soft noncommutative flag schemes' and their `soft subschemes'.
The details and an example are given in this section.

\bigskip

\begin{flushleft}
{\bf Soft noncommutative flag schemes}
\end{flushleft}
Let
 \begin{itemize}
  \item[$\LARGEdot$]
   $$
     \pi_0\;:\;
     \breve{R}_{\mathbf{I_0}}:= {\Bbb C}\langle M_{\mathbf{I}_0}(\breve{\boldz})\rangle
	   := {\Bbb C}\langle\breve{z}_{ij}\rangle_{ij}\;
       \longrightarrow\; R_{\mathbf{I}_0}={\Bbb C}[M_{\mathbf{I}_0}(\boldz)],
      \hspace{2em}\breve{z}_{ij}\;\longmapsto\; z_{ij},	
   $$
   be the (noncommutative) associative ${\Bbb C}$-algebra freely generated by the variable entries
     $\breve{z}_{ij}$'s of $M_{\mathbf{I}_0}(\breve{\boldz})$,
   with the built-in ${\Bbb C}$-algebra-epimorphism $\pi_0$ from commutatization
     that sends $\breve{z}_{ij}\in \breve{R}_{\mathbf{I}_0}$ to $z_{ij}\in R_{\mathbf{I}_0}$;
	
  \item[$\LARGEdot$]	
    $\breve{S}_{\mathbf{I}_0; \underline{\mathbf{I}\!\mathbf{I}}^\prime,\circ}$
	 be a lifting of  $S_{\mathbf{I}_0; \underline{\mathbf{I}\!\mathbf{I}}^\prime,\circ}$
	 in $\breve{R}_{\mathbf{I}_0}$ under $\pi_0$,
	 (i.e.\
    $\breve{S}_{\mathbf{I}_0; \underline{\mathbf{I}\!\mathbf{I}}^\prime,\circ}
	   \subset  \pi_0^{-1}(S_{\mathbf{I}_0; \underline{\mathbf{I}\!\mathbf{I}}^\prime,\circ})$
	  such that the restriction
	$\pi_0:
	    \breve{S}_{\mathbf{I}_0; \underline{\mathbf{I}\!\mathbf{I}}^\prime,\circ}
		 \rightarrow
		 S_{\mathbf{I}_0; \underline{\mathbf{I}\!\mathbf{I}}^\prime,\circ}$
	   is a set-isomorphism);   and
	
   \item[$\LARGEdot$]	
    $$
	  \breve{R}_{\underline{\mathbf{I}\!\mathbf{I}}}\;
	   :=\;  {\Bbb C}\langle M_{\mathbf{I}_0}(\breve{\boldz})\rangle
                 \langle
				   \breve{S}_{\mathbf{I}_0; \underline{\mathbf{I}\!\mathbf{I}}, \circ}^{\;-1} \rangle\;
	   :=\;  \frac{{\Bbb C}\langle \breve{z}_{ij}, \breve{w}_k\,|\, ij, k \rangle}
	         {(\breve{w}_k \breve{m}_k(\breve{\boldz})-1,
			         \breve{m}_k(\breve{\boldz}) \breve{w}_k -1 \,|\, k)\rule{0ex}{2ex}}\,.
	 $$
 \end{itemize}	
Here,
 \begin{itemize}
  \item[$\LARGEdot$]
  recall that
  $S_{\mathbf{I}_0; \underline{\mathbf{I}\!\mathbf{I}}^\prime,\circ}
       \subset {\Bbb C}[M_{\mathbf{I}_0}(\boldz)]$
   is the set  of the determinant of all distinguished submatrices of $M_{\mathbf{I}_0}(\boldz)$, now written as
   $\{m_k(\boldz)\,|\,
          k=1,\cdots, {n \choose d_1}{n\choose d_2-d_1} \cdots {n \choose d_r-d_{r-1}}  \}$;

   \item[$\LARGEdot$]
   $\breve{m}_k(\breve{\boldz})
       \in \breve{S}_{\mathbf{I}_0; \underline{\mathbf{I}\!\mathbf{I}}^\prime,\circ}$
      is the specified lifting of $m_k(\boldz)$;

    \item[$\LARGEdot$]	
   ${\Bbb C}\langle \breve{z}_{ij}, \breve{w}_k\, |\,ij, k \rangle
      = {\Bbb C}\langle \breve{z}_{ij}, \breve{w}_k \rangle_{ij, k}$
      is the associative ${\Bbb C}$-algebra
	  freely generated by the set $\{\breve{z}_{ij}, \breve{w}_k\}_{ij, k}$;    and
	
   \item[$\LARGEdot$]	
   $(\breve{w}_k \breve{m}_k(\breve{\boldz})-1,
			         \breve{m}_k(\breve{\boldz}) \breve{w}_k -1 \,|\, k)
		= (\breve{w}_k \breve{m}_k(\breve{\boldz})-1,
			         \breve{m}_k(\breve{\boldz}) \breve{w}_k -1)_k$
	   is the two-sided ideal of ${\Bbb C}\langle \breve{z}_{ij}, \breve{w}_k \rangle_{ij, k}$ 		
      	generated by the finite set of elements
		$\{\breve{w}_k \breve{m}_k(\breve{\boldz})-1,
			         \breve{m}_k(\breve{\boldz}) \breve{w}_k -1 \,|\, k\}$.		
 \end{itemize}
By construction, $\pi_0: \breve{R}_{\mathbf{I}_0}\rightarrow R_{\mathbf{I}_0}$ extends to
 a ${\Bbb C}$-algebra epimorphism from commutatization
 $$
  \pi_0\; :\;  \breve{R}_{\underline{\mathbf{I}\!\mathbf{I}}}\;
                     \longrightarrow\; R_{\underline{\mathbf{I}\!\mathbf{I}}}\,.
 $$

\smallskip

\begin{sdefinition} {\bf [noncommutative master ring for $\Fl(d_1, \cdots, d_r; n)$]}\; {\rm
 $\breve{R}_{\underline{\mathbf{I}\!\mathbf{I}}}$
 is called the {\it noncommutative master ring} for the flag variety $\Fl(d_1, \cdots, d_r; n)$.
}\end{sdefinition}

\medskip

\noindent
$\breve{R}_{\underline{\mathbf{I}\!\mathbf{I}}}$ plays the same role
 in the notion and construction of soft noncommutative flag schemes associated to $\Fl(d_1, \cdots, d_r; n)$
 as the group-algebra over ${\Bbb C}$ of a finitely-generated non-Abelian free group in the toric case,
cf.\ [L-Y2: Example 1.7] (D(15.1), NCS(1)).

\medskip

\begin{sdefinition} {\bf [soft noncommutative flag scheme]}\; {\rm
 A {\it $2^\underlineboldII$-system of ${\Bbb C}$-subalgebras of
             $\breve{R}_{\underline{\mathbf{I}\!\mathbf{I}}}$}
  is a collection $\{\breve{R}_{\underline{\mathbf{I}}}\}
             _{\underline{\mathbf{I}}\subset \underline{\mathbf{I}\!\mathbf{I}}}$
	 of ${\Bbb C}$-subalgebras of $\breve{R}_{\underline{\mathbf{I}\!\mathbf{I}}}$
   specified for each admissible chain
     $\underline{\mathbf{I}}\subset \underline{\mathbf{I}\!\mathbf{I}}$
   such that
     if $\underline{\mathbf{I}} \subset \underline{\mathbf{J}}$,
	 then $\breve{R}_{\underline{\mathbf{I}}}\subset \breve{R}_{\underline{\mathbf{J}}}$
	           in $\breve{R}_{\underline{\mathbf{I}\!\mathbf{I}}}$.
 A such system is called a {\it soft noncommutative flag scheme}
  if, in addition,
    the built-in ${\Bbb C}$-algebra epimorphism via commutatization
      $\pi_0: \breve{R}_{\underline{\mathbf{I}\!\mathbf{I}}}
	               \rightarrow R_{\underline{\mathbf{I}\!\mathbf{I}}}$
	  restricts a ${\Bbb C}$-algebra epimorphism
	  $\pi_0: \breve{R}_{\underline{\mathbf{I}}}\rightarrow R_{\underline{\mathbf{I}}}$
      for all $\underline{\mathbf{I}}\subset \underline{\mathbf{I}\!\mathbf{I}}$.
 We shall denote a soft noncommutative flag scheme associated to $\Fl(d_1, \cdots, d_r; n)$ by
  $$
    \breve{X}_{\underline{\mathbf{I}\!\mathbf{I}}}\;
	 :=\; ({\cal U}_0,  \breve{\cal O}_{{\cal U}_0}),
  $$
  where $\breve{\cal O}_{{\cal U}_0}$ is its {\it structure sheaf} adapted to the distinguished atlas,
   defined by the assignment
    $U_{\underline{\mathbf{I}}}\mapsto \breve{R}_{\underline{\mathbf{I}}}$,
	 $\underline{\mathbf{I}}\subset \underline{\mathbf{I}\!\mathbf{I}}$.
 In this case, 	for $\underline{\mathbf{I}}\subset \underline{\mathbf{J}}$,
   we shall denote the inclusion map
      $\breve{R}_{\underline{\mathbf{I}}}
	        \hookrightarrow \breve{R}_{\underline{\mathbf{J}}}$
	  of ${\Bbb C}$-algebras
	  by $\iota^\sharp_{\underline{\mathbf{J}}\underline{\mathbf{I}}}$,
	  where
	    $\iota_{\underline{\mathbf{J}}\underline{\mathbf{I}}}:
	       U_{\underline{\mathbf{J}}}\hookrightarrow U_{\underline{\mathbf{I}}}$
		 is the inclusion of charts in ${\cal U}_0$.
}\end{sdefinition}

\medskip

\begin{sdefinition} {\bf [softening of noncommutative flag scheme]}\; {\rm
 Let
   $\breve{X}_{\underline{\mathbf{I}\!\mathbf{I}}}
       =({\cal U}_0, \breve{\cal O}_{{\cal U}_0})$    and
   $\breve{X}^\prime_{\underline{\mathbf{I}\!\mathbf{I}}}
      = ({\cal U}_0, \breve{\cal O}^{\,\prime}_{{\cal U}_0})$
  be two soft noncommutative flag schemes associated to $\Fl(d_1, \cdots, d_r; n)$.
 We say that $\breve{X}^\prime_{\underline{\mathbf{I}\!\mathbf{I}}}$
  is a {\it softening} of $\breve{X}_{\underline{\mathbf{I}\!\mathbf{I}}}$
     or interchangeably that
	  $\breve{X}^\prime_{\underline{\mathbf{I}\!\mathbf{I}}}$
               {\it softens} $\breve{X}_{\underline{\mathbf{I}\!\mathbf{I}}}$
  if ${\cal O}_{{\cal U}_0}(U_{\underline{\mathbf{I}}})
         \subset {\cal O}^{\,\prime}_{{\cal U}_0}(U_{\underline{\mathbf{I}}})$
	 for all $\underline{\mathbf{I}} \subset \underline{\mathbf{I}\!\mathbf{I}}$.
 This defines a morphism
    $\breve{X}^\prime_{\underline{\mathbf{I}\!\mathbf{I}}}
	   \rightarrow \breve{X}_{\underline{\mathbf{I}\!\mathbf{I}}}$	
	of soft noncommutative flag schemes, named a {\it softening morphism}.
}\end{sdefinition}

\medskip

\begin{slemma} {\bf [common softening]}\;
 Let
   $\breve{X}_\underlineboldII$ and $\breve{X}^\prime_\underlineboldII$
   be two soft noncommutative flag schemes associated to $\Fl(d_1, \cdots, d_r; n)$.
 Then there exists another soft noncommutative flag scheme $\breve{X}^{\prime\prime}_\underlineboldII$
    associated to $\Fl(d_1, \cdots, d_r; n)$
  that softens both $\breve{X}_\underlineboldII$ and $\breve{X}^{\prime\prime}_\underlineboldII$.
\end{slemma}

\begin{proof}
 For example, take ${\cal O}^{\,\prime\prime}_{{\cal U}_0}(U_\underlineboldI)$
   to be the ${\Bbb C}$-subalgebra of $\breve{R}_\underlineboldII$
   generated by
   ${\cal O}_{{\cal U}_0}(U_\underlineboldI)
        \cup {\cal O}^{\,\prime}_{{\cal U}_0}(U_\underlineboldI)$
   for $\underlineboldI \subset \underlineboldII$.
\end{proof}

\bigskip

\begin{flushleft}
{\bf Construction of soft noncommutative flag schemes}
\end{flushleft}
In the toric case, there are monoidal structures that one would like to keep track,
  making the construction of a soft noncommutative toric scheme slightly involved.
Here, there is no similar built-in structure from the flag variety $\Fl(d_1, \cdots, d_r; n)$.
Thus the task of the construction becomes indeed light.

Recall the $2^\underlineboldII$-system $X_\underlineboldII =({\cal U}_0, {\cal O}_{{\cal U}_0})$
  of rings that represents $\Fl(d_1, \cdots, d_r; n)$ contravariantly.
A soft noncommutative flag scheme
  $\breve{X}_\underlineboldII =({\cal U}_0,  \breve{\cal O}_{{\cal U}_0})$
  that is obtained by soft-gluings of
    $|\underlineboldII|$-many copies of noncommutative affine spaces over $\Bbb C$,
	        cf.\ [L-Y2: Definition 2.1.1] (D(15.1), NCS(1)),
  can be constructed as follows:
 \begin{itemize}
  \item[$\LARGEdot$]
  For $\mathbf{I}\in \underlineboldII$ an admissible sequence,
   let $G_\mathbf{I}\subset R_{\mathbf{I}_0}[S_{\mathbf{I}_0;\mathbf{I}}^{\;-1}]$
    be the set  of the entries of 	the $r\times n$ matrix
	$C_{\mathbf{I_0}; \mathbf{I}} M_{\mathbf{I}_0}(\boldz)$
	that are not $0$ or $1$.
   Recall that $R_\mathbf{I}\subset R_\underlineboldII$ is the polynomial ring over ${\Bbb C}$
    generated by $G_\mathbf{I}$.
  Let $\breve{G}_\mathbf{I}$ be a lifting of $G_\mathbf{I}$ in $\breve{R}_\underlineboldII$
    via $\pi_0$; i.e., $\breve{G}_\mathbf{I}\subset \pi_0^{-1}(G_\mathbf{I})$ such that
	the restriction $\pi_0: \breve{G}_\mathbf{I}\rightarrow G_\mathbf{I}$ is a set-isomorphism.
  Then, set
   $$
    \begin{array}{rcl}
	 \breve{\cal O}_{{\cal U}_0}(U_\mathbf{I})\;
	  =\; \breve{R}_\mathbf{I}&  = &
	   \mbox{the ${\Bbb C}$-subalgebra of
	                   $\breve{R}_{\underline{\mathbf{I}\!\mathbf{I}}}$
                    generated by $\breve{G}_\mathbf{I}$.}
    \end{array}					
   $$
   
  \item[$\LARGEdot$]
  For $\underline{\mathbf{I}}=\{\mathbf{I}_1, \cdots, \mathbf{I}_l\}\subset \underlineboldII$
   an admissible chain, set
  $$
   \breve{\cal O}_{{\cal U}_0}(U_\underlineboldI)\;
    =\; \breve{R}_{\underline{\mathbf{I}}}\;
    =\; \mbox{the ${\Bbb C}$-subalgebra of $\breve{R}_{\underline{\mathbf{I}\!\mathbf{I}}}$
	      generated by $\breve{R}_{\mathbf{I}_1} \cup \cdots \cup\breve{R}_{\mathbf{I}_l}$.}
  $$
 \end{itemize}
By construction,
 for $\underline{\mathbf{I}}\subset \underline{\mathbf{J}}$,
 one has $\breve{R}_\underlineboldI \subset \breve{R}_{\underline{\mathbf{J}}}$.
Thus, the system defines a soft noncommutative flag scheme over ${\Bbb C}$.
Furthermore,
 for $\underlineboldI\subset \underlineboldII$,
  $\pi_0:\breve{R}_\underlineboldI \rightarrow R_\underlineboldI$
   is the commutatization of $\breve{R}_\underlineboldI$    and,
 for $\underlineboldI\subset \underline{\mathbf{J}}$,
  the following diagrams commute
  $$
   \xymatrix{
    \breve{R}_{\underline{\mathbf{I}}}\;
	  \ar@{^{(}->}[rr]^-{\iota^\sharp_{\underline{\mathbf{J}}\underline{\mathbf{I}}}}
	  \ar@{->>}[d]_-{\pi_0}
	   && \breve{R}_{\underline{\mathbf{J}}} \ar@{->>}[d]^-{\pi_0}	
	   \\	
    R_{\underline{\mathbf{I}}}\;
	  \ar@{^{(}->}[rr]^-{\iota^\sharp_{\underline{\mathbf{J}}\underline{\mathbf{I}}}}
	   && R_{\underline{\mathbf{J}}} 	  &\hspace{-4em}.
   }
  $$
Thus, $\Fl(d_1, \cdots, d_r; n)\hookrightarrow \breve{X}_{\underlineboldII}$
 as a maximal commutative subscheme.

\bigskip

\begin{flushleft}
{\bf Closed subschemes of $\breve{X}_\underlineboldII$}
\end{flushleft}
For a closed subscheme $Z\subset X$, one can realize $Z$ as a $2^\underlineboldII$-system
 $$
   Z_\underlineboldII\; :=\; ({\cal U}_0^Z, {\cal O}_{{\cal U}_0^Z} ),
 $$
 where
  ${\cal U}_0^Z:=\{Z\cap U_\underlineboldI \}_{\underlineboldI\subset \underlineboldII}$
     is the {\it induced distinguished atlas} on $Z$   and
  ${\cal O}_{{\cal U}_0^Z}$  is the {\it structure sheaf} of $Z_\underlineboldII$,
      defined by the specification
      ${\cal O}_{{\cal U}_0^Z}(U_{\underlineboldI})=R_\underlineboldI/I_\underlineboldI$,
        where $I_\underlineboldI$ is the ideal of $R_\underlineboldI$ associated to the closed subscheme
       $Z\cap U_\underlineboldI $ of $U_\underlineboldI$, for $\underlineboldI \subset \underlineboldII$.
One may write also 	
  ${\cal O}_{{\cal U}_0^Z}= {\cal O}_{{\cal U}_0}/{\cal I}_Z$,
  where ${\cal I}_Z$ is the ideal sheaf of $Z\subset X$.

Let $\breve{I}_\underlineboldI\subset \breve{R}_\underlineboldI$
   be a two-sided ideal, for $\underlineboldI \subset \underlineboldII$,
   such that
    $\pi_0(\breve{I}_\underlineboldI) = I_\underlineboldI$ for all $\underlineboldI$   and
   that $\iota_{\underline{\mathbf{J}}\underlineboldI}^\sharp(\breve{I}_\underlineboldI)
                            \subset \breve{I}_{\underline{\mathbf{J}}}$
         for $\underlineboldI \subset \underline{\mathbf{J}}$.
Then,
 the ${\Bbb C}$-algebra-homomorphism
  $\iota_{\underline{\mathbf{J}}\underlineboldI}^\sharp :
     \breve{R}_\underlineboldI \rightarrow \breve{R}_{\underline{\mathbf{J}}}$
 descend to quotients (with slight abuse of notation)
  $$
    \iota_{\underline{\mathbf{J}}\underlineboldI}^\sharp\; :\;
     \breve{R}_\underlineboldI/\breve{I}_\underlineboldI\;
	 \rightarrow\;  \breve{R}_{\underline{\mathbf{J}}}/\breve{I}_{\underline{\mathbf{J}}}
  $$
   for $\underlineboldI \subset \underline{\mathbf{J}}$.
	
\medskip

\begin{sdefinition}
{\bf [soft noncommutative closed subscheme of $\breve{X}_\underlineboldII$ associated to $Z\subset X$]}\;
{\rm
 The $2^\underlineboldII$-system of quotient ${\Bbb C}$-algebras
   $$
     \breve{Z}_\underlineboldII\; :=\;
       ({\cal U}_0^Z, \breve{\cal O}_{{\cal U}_0^Z})	
   $$
 as above is called
 a {\it soft noncommutative closed subscheme of $\breve{X}_\underlineboldII$ associated to $Z\subset X$}.
 Here,
   $\breve{\cal O}_{{\cal U}_0^Z}$  is the {\it structure sheaf}
   of $\breve{Z}_\underlineboldII$, defined by the specification
   $\breve{\cal O}_{{\cal U}_0^Z}(U_\underlineboldI)
      = \breve{R}_\underlineboldI / \breve{I}_\underlineboldI$
   for $\underlineboldI\subset \underlineboldII$.
}\end{sdefinition}

\medskip

\begin{slemma} {\bf [existence]}\;
 Let
   $Z$ be a closed subscheme of $\Fl(d_1, \cdots, d_r; n)$   and
   $\breve{X}_\underlineboldII$ be a soft noncommutative scheme associated to $\Fl(d_1, \cdots, d_r; n )$.
 Then there exists a soft noncommutative closed subscheme $\breve{Z}_\underlineboldII$
   of $\breve{X}_\underlineboldII$ associated to $Z$.
\end{slemma}

\begin{proof}
 Continuing the notation in the current theme.
 Note that the ideal $I_\mathbf{I}$ associated to $Z \cap U_\mathbf{I} \subset U_\mathbf{I}$
   is finitely generated, for $\mathbf{I}\in \underlineboldII$.
 Let $\breve{I}_\mathbf{I}$ be the two-sided ideal of $\breve{R}_\mathbf{I}$
   generated by a lifting of  a finite generating set of $I_\mathbf{I}$ to $\breve{R}_\mathbf{I}$.
 For general $\underlineboldI =\{\mathbf{I}_1, \cdots, \mathbf{I}_l\}\subset \underlineboldII$,
  let $\breve{I}_\underlineboldI$ be the two-sided ideal in $\breve{R}_\underlineboldI$
   generated by $\sum_{j=1}^l \breve{I}_{\mathbf{I}_j} \subset \breve{R}_\underlineboldI$.
 
\end{proof}

\bigskip

\begin{flushleft}
{\bf Example: $\Gr(2; 4)$}
\end{flushleft}
The Grassmann variety $\Gr(2; 4)$ of $2$-planes in ${\Bbb C}^4$
  is the simplest flag variety that is not a projective space itself.
The Pl\"{u}cker embedding realizes it as a nonsingular quadric hypersurface in $\CP^5$,
  whose intersection with a generic quartic hypersurface  is a Calabi-Yau $3$-fold $Z$, (e.g.\ [B-CF-K-vS]).
The distinguished atlas ${\cal U}_0$ for $\Gr(2;4)$ contains ${4 \choose 2}=6$ maximal charts.
 \begin{itemize}
  \item[$\LARGEdot$] The {\it reference chart} $U_{\mathbf{I}_0}=U_{\{1,2\}}$:\\
   $R_{\{1, 2\}}$ is the polynomial ring over ${\Bbb C}$ generated by the non-$0$-nor-$1$ entries of
   $$
     M_{\{1,2\}}(\boldz)\;=\;
	  \left[
	   \begin{array}{cccc}
	    1 & 0 & z_{13}  & z_{14} \\[.6ex]
		0 & 1 & z_{23}  & z_{24}	
	   \end{array}
	  \right].
   $$
 \end{itemize}
The {\it master ring} for $\Gr(2; 4)$  is given by the following localization of $R_{\{1, 2\}}$:
 $$
  R_{\underlineboldII}\;=\;
   {\Bbb C}[z_{13}, z_{14}, z_{23}, z_{24}]
                         [z_{13}^{-1}, z_{14}^{-1}, z_{23}^{-1}, z_{24}^{-1},
						  (z_{13}z_{24}-z_{14}z_{23})^{-1}] .
 $$
The function ring of other maximal charts in ${\cal U}_0$, as a ${\Bbb C}$-subalgebra of $R_\underlineboldII$,
 are given below.
 \begin{itemize}
  \item[$\LARGEdot$] $U_{\{1, 3\}}$:\hspace{2em}
   $R_{\{1, 3\}}$ is the polynomial ring over ${\Bbb C}$ generated by the non-$0$-nor-$1$ entries of
   $$
     C_{\{1, 2\}; \{1, 3\}}M_{\{1,2\}}(\boldz)\;=\;
	  \left[
	   \begin{array}{cccc}
	    1 & - z_{23}^{-1} z_{13}  & 0   & z_{14}- z_{23}^{-1}z_{13} z_{24} \\[.6ex]
		0 &   z_{23}^{-1}    &  1    &  z_{23}^{-1}z_{24}	
	   \end{array}
	  \right].
   $$
  
  \item[$\LARGEdot$] $U_{\{1, 4\}}$:\hspace{2em}
   $R_{\{1, 4\}}$ is the polynomial ring over ${\Bbb C}$ generated by the non-$0$-nor-$1$ entries of
   $$
     C_{\{1, 2\}; \{1, 4\}}M_{\{1,2\}}(\boldz)\;=\;
	  \left[
	   \begin{array}{cccc}
	    1 & - z_{24}^{-1} z_{14}  & z_{13}- z_{24}^{-1}z_{14} z_{23} & 0 \\[.6ex]
		0 &   z_{24}^{-1}    &  z_{24}^{-1}z_{23}	 & 1
	   \end{array}
	  \right].
   $$
  
  \item[$\LARGEdot$] $U_{\{2, 3\}}$:\hspace{2em}
   $R_{\{2, 3\}}$ is the polynomial ring over ${\Bbb C}$ generated by the non-$0$-nor-$1$ entries of
   $$
     C_{\{1, 2\}; \{2, 3\}}M_{\{1,2\}}(\boldz)\;=\;
	  \left[
	   \begin{array}{cccc}
	     - z_{13}^{-1} z_{23}  & 1  & 0   & z_{24}- z_{13}^{-1}z_{14} z_{23} \\[.6ex]
		   z_{13}^{-1}                      & 0 &  1   &  z_{13}^{-1}z_{14}	
	   \end{array}
	  \right].
   $$

  \item[$\LARGEdot$] $U_{\{2, 4\}}$:\hspace{2em}
   $R_{\{2, 4\}}$ is the polynomial ring over ${\Bbb C}$ generated by the non-$0$-nor-$1$ entries of
   $$
     C_{\{1, 2\}; \{1, 3\}}M_{\{1,2\}}(\boldz)\;=\;
	  \left[
	   \begin{array}{cccc}
	    - z_{14}^{-1} z_{24}  & 1   & z_{23}- z_{14}^{-1}z_{13} z_{24} & 0 \\[.6ex]
		   z_{14}^{-1}    &  0    &  z_{14}^{-1}z_{13}  & 1	
	   \end{array}
	  \right].
   $$
   
  \item[$\LARGEdot$] $U_{\{3,4\}}$:\hspace{2em}
   $R_{\{3,4\}}$ is the polynomial ring over ${\Bbb C}$ generated by the non-$0$-nor-$1$ entries of
   $$
     C_{\{1, 2\}; \{3,4\}}M_{\{1,2\}}(\boldz)\;=\;
	  \left[
	   \begin{array}{cccc}
	    (z_{13}z_{24}-z_{14}z_{23})^{-1}z_{24}
              &  -(z_{13}z_{24}-z_{14}z_{23})^{-1}z_{14}   & 1    & 0		\\[.6ex]
	     -(z_{13}z_{24}-z_{14}z_{23})^{-1}z_{23}	
		      &  (z_{13}z_{24}-z_{14}z_{23})^{-1}z_{13}    & 0    & 1
	   \end{array}
	  \right].
   $$
 \end{itemize}
 
 Without further input from mathematics or D-brane probe in string theory,
  a soft noncommutative flag scheme can be obtained very flexibly or randomly.
 For example, here for $\Gr(2; 4)$,
  a formal replacement
  $$
    z_{ij}\; \longmapsto\; \breve{z}_{ij}
  $$
  while keeping all the expressions above intact with the knowing that $\breve{z}_{ij}$'s are noncommutative
 gives rise to a soft noncommutative Grassmann scheme $\breve{X}_\underlineboldII$
 associated to $\Gr(2;4)$ by the recipe in the previous theme.

\bigskip

\begin{flushleft}
{\bf Soft noncommutative Calabi-Yau schemes in $\breve{X}_\underlineboldII$ associated to $\Gr(2;4)$?}
\end{flushleft}
Let
\begin{itemize}
 \item[$\LARGEdot$]
  $\mathbf{P}^5$ be the projective $5$- space over ${\Bbb C}$
       with the homogeneous coordinates $[y_0: y_1: \cdots : y_5]$;
	
 \item[$\LARGEdot$]	
  $V_i \simeq \mathbf{A}^5$ (the affine $5$-space over ${\Bbb C}$)
    be  the complement of the hyperplane $\{y_i=0\}$ in $\mathbf{P}^5$, for $i=0, \cdots, 5$,
    that together produce an atlas on $\mathbf{P}^5$;
	
 \item[$\LARGEdot$]
  $\varPsi: \Gr(2; 4)\hookrightarrow \mathbf{P}^5$
   be the Pl\"{u}cker embedding of $\Gr(2; 4)$ in $\mathbf{P}^5$ given by the correspondence
    $$
	   M\; \longmapsto\;
	      [|M^{\{1,2\}}| : |M^{\{1,3\}}| : |M^{\{1, 4\}}| :
		   |M^{\{2, 3\}}| : |M^{\{2,4\}}| : |M^{\{3,4\}}|],
	$$
 	where
	  $M$ is a $2\times 4$ matrix with coefficients in ${\Bbb C}$ and
	  $|M^\mathbf{I}|$, $\mathbf{I}\in\underlineboldII$,
	              is the determinant of the $2\times 2$ submatrix of $M$ specified by $\mathbf{I}$.
\end{itemize}				
Then, in terms of $X_{\underlineboldII}$,
$\varPsi: \Gr(2; 4)\rightarrow \mathbf{P}^5$ is given by a gluing of  embeddings of affine schemes
 $$
  \begin{array}{ccc}
    \xymatrix{U_{\{1,2\}}\ar@{^{(}->}[rr]^-{\varPhi_{\{1, 2\}}}  && V_0}\!,
    &  \xymatrix{U_{\{1,3\}}\ar@{^{(}->}[rr]^-{\varPhi_{\{1, 3\}}}  && V_1}\!,
    &	 \xymatrix{U_{\{1,4\}}\ar@{^{(}->}[rr]^-{\varPhi_{\{1, 4\}}}  && V_2}\!,
		 \\[1.2ex]
    \xymatrix{U_{\{2, 3\}}\ar@{^{(}->}[rr]^-{\varPhi_{\{2, 3\}}}  && V_3}\!,
	& \xymatrix{U_{\{2, 4\}}\ar@{^{(}->}[rr]^-{\varPhi_{\{2, 4\}}}  && V_4}\!,
	& \xymatrix{U_{\{3, 4\}}\ar@{^{(}->}[rr]^-{\varPhi_{\{3, 4\}}}  && V_5}\!,	
  \end{array}
 $$
 and
the underlying
   $\varPsi^\sharp: {\cal O}_{\mathbf{P}^5}\rightarrow \varPsi_\ast {\cal  O}_{Gr(2;4)}$
 is a gluing of ${\Bbb C}$-algebra-homomorphisms
 $$
  \begin{array}{ccc}
   \xymatrix{{\cal O}_{\mathbf{P}^5}(V_0)\ar[rr]^-{\varPhi^\sharp_{\{1, 2\}}}
                           && R_{\{1, 2\}}}\!,
   & \xymatrix{{\cal O}_{\mathbf{P}^5}(V_1)\ar[rr]^-{\varPhi^\sharp_{\{1, 3\}}}
                           && R_{\{1, 3\}}}\!,
   & \xymatrix{{\cal O}_{\mathbf{P}^5}(V_2)\ar[rr]^-{\varPhi^\sharp_{\{1, 4\}}}
                           && R_{\{1, 4\}}}\!,
          \\[1.2ex]
   \xymatrix{{\cal O}_{\mathbf{P}^5}(V_3)\ar[rr]^-{\varPhi^\sharp_{\{2, 3\}}}
                           && R_{\{2, 3\}}}\!,
   & \xymatrix{{\cal O}_{\mathbf{P}^5}(V_4)\ar[rr]^-{\varPhi^\sharp_{\{2, 4\}}}
                           && R_{\{2, 4\}}}\!,
   & \xymatrix{{\cal O}_{\mathbf{P}^5}(V_5)\ar[rr]^-{\varPhi^\sharp_{\{3, 4\}}}
                           && R_{\{3, 4\}}}\!.
  \end{array}
 $$
With slight abuse of notations,
 these $\varPsi_\mathbf{I}^\sharp$, $\mathbf{I}\in \underlineboldII$,
  can be described by ${\Bbb C}$-vector-space-homomorphisms
  $$
     \varPsi_\mathbf{I}^\sharp\; :\;
	  \varGamma(\mathbf{P}^5,  {\cal O}_{\mathbf{P}^5}(n))\;
	   \longrightarrow\; R_\mathbf{I},
  $$
  $n=1, 2, \cdots\,$.
Explicitly, for a homogeneous polynomial of degree $n$
 $$
    f\;  :=\;  f(y_0, y_1, y_2, y_3, y_4, y_5)\;
	   \in\;    \varGamma(\mathbf{P}^5,  {\cal O}_{\mathbf{P}^5}(n)),
 $$
one has
 $${\scriptsize\hspace{1ex}
  \begin{array}{rcl}
   \varPhi_{\{1, 2\}}^\sharp(f) & =  &
    f(|M_{\{1, 2\}}(\boldz)^{\{1, 2\}}|, |M_{\{1, 2\}}(\boldz)^{\{1, 3\}}|,
         |M_{\{1, 2\}}(\boldz)^{\{1, 4\}}|, |M_{\{1, 2\}}(\boldz)^{\{2, 3\}}|,    	
		 |M_{\{1, 2\}}(\boldz)^{\{2, 4\}}|, |M_{\{1, 2\}}(\boldz)^{\{3, 4\}}|  )  \\[1.2ex]
	 & = &
	 f(1,\; z_{23},\; z_{24},\; -z_{13},\; -z_{14},\; z_{13}z_{24}- z_{14}z_{23}),
	   \\[.6ex]
  \end{array}
  }
 $$
 $${\scriptsize
  \begin{array}{rcl}
   \varPhi_{\{1, 3\}}^\sharp(f) & =  &
    f(|C_{\{1, 2\}; \{1, 3\}}M_{\{1, 2\}}(\boldz)^{\{1, 2\}}|,\;
	     |C_{\{1, 2\}; \{1, 3\}}M_{\{1, 2\}}(\boldz)^{\{1, 3\}}|,\;
         |C_{\{1, 2\}; \{1, 3\}}M_{\{1, 2\}}(\boldz)^{\{1, 4\}}|,
		 \\[.6ex]
		 && \hspace{6em}
		 |C_{\{1, 2\}; \{1, 3\}}M_{\{1, 2\}}(\boldz)^{\{2, 3\}}|,\;    	
		 |C_{\{1, 2\}; \{1, 3\}}M_{\{1, 2\}}(\boldz)^{\{2, 4\}}|,\;
		 |C_{\{1, 2\}; \{1, 3\}}M_{\{1, 2\}}(\boldz)^{\{3, 4\}}|  )
		 \\[1.2ex]
	 & = &
	 f(z_{23}^{-1},\; 1,\; z_{23}^{-1}z_{24},\; -z_{23}^{-1}z_{13},\;
	      -z_{23}^{-1}z_{14},\;  -z_{14}+ z_{23}^{-1}z_{13}z_{24}),
		\\[.6ex]
  \end{array}
  }
 $$
 $${\scriptsize
  \begin{array}{rcl}
   \varPhi_{\{1, 4\}}^\sharp(f) & =  &
    f(|C_{\{1, 2\}; \{1, 4\}}M_{\{1, 2\}}(\boldz)^{\{1, 2\}}|,\;
	     |C_{\{1, 2\}; \{1, 4\}}M_{\{1, 2\}}(\boldz)^{\{1, 3\}}|,\;
         |C_{\{1, 2\}; \{1, 4\}}M_{\{1, 2\}}(\boldz)^{\{1, 4\}}|,
		 \\[.6ex]
		 && \hspace{6em}
		 |C_{\{1, 2\}; \{1, 4\}}M_{\{1, 2\}}(\boldz)^{\{2, 3\}}|,\;    	
		 |C_{\{1, 2\}; \{1, 4\}}M_{\{1, 2\}}(\boldz)^{\{2, 4\}}|,\;
		 |C_{\{1, 2\}; \{1, 4\}}M_{\{1, 2\}}(\boldz)^{\{3, 4\}}|  )
		 \\[1.2ex]
	 & = &
	 f(z_{24}^{-1},\; z_{24}^{-1}z_{23},\; 1,\; -z_{24}^{-1}z_{13},\;
	      -z_{24}^{-1}z_{14},\;  z_{13}- z_{24}^{-1}z_{14}z_{23}),
		  \\[.6ex]
  \end{array}
  }
 $$
 $${\scriptsize
  \begin{array}{rcl}
   \varPhi_{\{2, 3\}}^\sharp(f) & =  &
    f(|C_{\{1, 2\}; \{2, 3\}}M_{\{1, 2\}}(\boldz)^{\{1, 2\}}|,\;
	     |C_{\{1, 2\}; \{2, 3\}}M_{\{1, 2\}}(\boldz)^{\{1, 3\}}|,\;
         |C_{\{1, 2\}; \{2, 3\}}M_{\{1, 2\}}(\boldz)^{\{1, 4\}}|,
		 \\[.6ex]
		 && \hspace{6em}
		 |C_{\{1, 2\}; \{2, 3\}}M_{\{1, 2\}}(\boldz)^{\{2, 3\}}|,\;    	
		 |C_{\{1, 2\}; \{2, 3\}}M_{\{1, 2\}}(\boldz)^{\{2, 4\}}|,\;
		 |C_{\{1, 2\}; \{2, 3\}}M_{\{1, 2\}}(\boldz)^{\{3, 4\}}|  )
		 \\[1.2ex]
	 & = &
	 f(-z_{13}^{-1},\; -z_{13}^{-1}z_{23},\; -z_{13}^{-1}z_{24},\; 1,\;
	       z_{13}^{-1}z_{14},\;  -z_{24}+ z_{13}^{-1}z_{14}z_{23}),
		 \\[.6ex]
  \end{array}
  }
 $$
 $${\scriptsize
  \begin{array}{rcl}
   \varPhi_{\{2, 4\}}^\sharp(f) & =  &
    f(|C_{\{1, 2\}; \{2, 4\}}M_{\{1, 2\}}(\boldz)^{\{1, 2\}}|,\;
	     |C_{\{1, 2\}; \{2, 4\}}M_{\{1, 2\}}(\boldz)^{\{1, 3\}}|,\;
         |C_{\{1, 2\}; \{2, 4\}}M_{\{1, 2\}}(\boldz)^{\{1, 4\}}|,
		 \\[.6ex]
		 && \hspace{6em}
		 |C_{\{1, 2\}; \{2, 4\}}M_{\{1, 2\}}(\boldz)^{\{2, 3\}}|,\;    	
		 |C_{\{1, 2\}; \{2, 4\}}M_{\{1, 2\}}(\boldz)^{\{2, 4\}}|,\;
		 |C_{\{1, 2\}; \{2, 4\}}M_{\{1, 2\}}(\boldz)^{\{3, 4\}}|  )
		 \\[1.2ex]
	 & = &
	 f(-z_{14}^{-1},\; -z_{14}^{-1}z_{23},\; -z_{14}^{-1}z_{24},\; z_{14}^{-1}z_{13},\;
	      1,\;   z_{23}- z_{14}^{-1}z_{13}z_{24}),
  \end{array}
  }
 $$
 $${\scriptsize
  \begin{array}{rcl}
   \varPhi_{\{3, 4\}}^\sharp(f) & =  &
    f(|C_{\{1, 2\}; \{3, 4\}}M_{\{1, 2\}}(\boldz)^{\{1, 2\}}|,\;
	     |C_{\{1, 2\}; \{3, 4\}}M_{\{1, 2\}}(\boldz)^{\{1, 3\}}|,\;
         |C_{\{1, 2\}; \{3, 4\}}M_{\{1, 2\}}(\boldz)^{\{1, 4\}}|,
		 \\[.6ex]
		 && \hspace{6em}
		 |C_{\{1, 2\}; \{3, 4\}}M_{\{1, 2\}}(\boldz)^{\{2, 3\}}|,\;    	
		 |C_{\{1, 2\}; \{3, 4\}}M_{\{1, 2\}}(\boldz)^{\{2, 4\}}|,\;
		 |C_{\{1, 2\}; \{3, 4\}}M_{\{1, 2\}}(\boldz)^{\{3, 4\}}|  )
		 \\[1.2ex]
	 & = &
	 f((z_{13}z_{24}- z_{14}z_{23})^{-1},\;
	      (z_{13}z_{24}- z_{14}z_{23})^{-1}z_{23},\;
		  (z_{13}z_{24}- z_{14}z_{23})^{-1} z_{24},
		    \\[.6ex]
			&& \hspace{6em}
		  -(z_{13}z_{24}- z_{14}z_{23})^{-1}z_{13},\;
		  -(z_{13}z_{24}- z_{14}z_{23})^{-1}z_{14},\;
		      1).
  \end{array}
  }
 $$
Each $\varPsi_\mathbf{I}^\sharp(f)$  can be further lifted to $\breve{R}_\mathbf{I}$
  under $\pi_0: \breve{R}_\mathbf{I}\rightarrow R_\mathbf{I}$,
 for $\mathbf{I}\in \underlineboldII$.
Together with the proof of Lemma~2.6,
                                          % Lemma [existence]
 this gives a direct association of a subscheme $V$ of $\mathbf{P}^5$
 with a soft noncommutative closed subscheme
      $\breve{Z}_\underlineboldII$ of $\breve{X}_\underlineboldII$
	  associated to $V\cap \Gr(2;4)$.
	
In particular,
 for $Z$ a subscheme in $\Gr(2; 4)$ arising from a generic quartic hypersurface in $\mathbf{P}^5$
   described by a homogeneous polynomial
   $f\in \varGamma(\mathbf{P}^5, {\cal O}_{{\mathbf{P}}^5}(4))$,
one has a soft noncommutative scheme $\breve{Z}_\underlineboldII$
 that contains the Calabi-Yau $3$-fold $Z$ as its master commutative subscheme.

We conclude the current theme and the notes with a guiding question from [L-Y2]:

\medskip

%question
\noindent
{\bf Question [soft noncommutative Calabi-Yau space and mirror]}\; 
 ([L-Y2: Question 3.12] (D(15.1), NCS(1)).)
 What is the correct notion/definition of {\it soft noncommutative Calabi-Yau spaces}
   in the current context?
 From pure mathematical generalization of the commutative case?
 From world-volume conformal invariance or supersymmetry of D-branes?
 What is the {\it mirror symmetry} phenomenon in this context?																	
% end-question
				
\medskip

\noindent
and {\sc Figure}~2-1.
    %
	% \marginpar{\raggedright\tiny $\bullet$ {\sc Figure}:\\  soft-noncommutative-scheme.pdf}

\vspace{-2.4em}	
\begin{figure}[htbp]
 \bigskip
  \centering
  \includegraphics[width=0.80\textwidth]{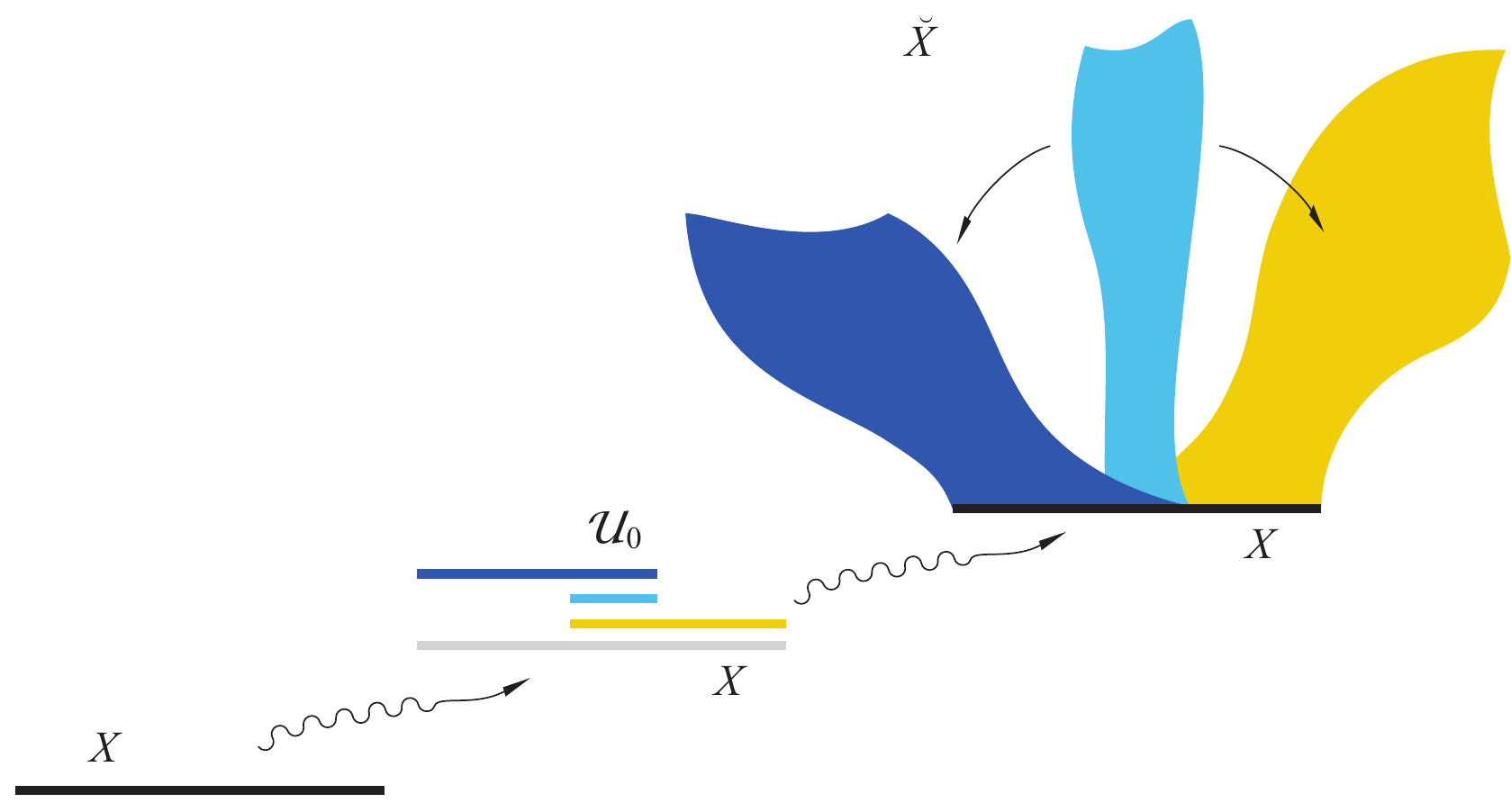}
  
  \bigskip
  
 \centerline{\parbox{13cm}{\small\baselineskip 12pt
  {\sc Figure}~2-1.
  The world-volume of stacked D-branes naturally carries an Azumaya structure, which is noncommutative,
    cf.\ [L-Y1] (D(1)), [Liu]. This allows stacked D-branes to get mapped to noncommutative target-spaces.
  On the other hand,
   a theory of localizations of general noncommutative rings is problematic,
   which leads to the notion of noncommutative schemes that directly generalizes Grothendieck's work for
   (Commutative) Algebraic Geometry impossible.
  Yet, for the purpose of serving as target spaces of dynamical D-branes, one only needs a gluing system of charts.
  This motivates the notion of `{\it soft noncommutative schemes}'.
  Beginning with a (commutative) scheme $X$, we may cover it by an atlas ${\cal U}_0$
    of simple enough, reasonably good affine charts.
  One then extends this inclusion system of charts to a gluing system system
  $\breve{X}:= ({\cal U}_0, \breve{\cal O}_{{\cal U}_0})$
   of noncommutative charts,
   in which inclusions of commutative charts extend but relax to morphisms $\longrightarrow$ of noncommutative charts
    that satisfy cocycle conditions for gluing. 
 By construction, $X$ naturally embeds in $\breve{X}$.
 $\breve{X}$ as constructed can be ``very fluffy around $X$".
 Other inputs from String Theory, particularly D-brane probes, 
  should be taken into account to make a good notion of `soft noncommutative Calabi-Yau schemes' 
  when $X$ is a Calabi-Yau space.
  }}% end-small % end-centerline
\end{figure}

%\vspace{2em}				
\newpage
\baselineskip 13pt
%references
{\footnotesize

\vspace{1em}

\noindent
chienhao.liu@gmail.com, 
chienliu@cmsa.fas.harvard.edu; \\
  % vafa@physics.harvard.edu;\\
yau@math.harvard.edu

}%endfootnotesize

\end{document}